\documentclass[12pt,longbibliography]{article}

\usepackage[utf8]{inputenc}

\usepackage[T1]{fontenc}
\usepackage{changepage}

\usepackage[a4paper, margin=2.5cm]{geometry}

\usepackage{amsmath, amsthm, amsfonts, amssymb}
\usepackage{mathrsfs}           
\usepackage{multicol}
\usepackage[all]{xy}
\usepackage{graphicx}
\usepackage[usenames,dvipsnames]{xcolor}

\usepackage[colorlinks=true,linkcolor=black,citecolor=black,urlcolor=black,breaklinks]{hyperref}

\usepackage{bbm} 

\DeclareMathOperator*{\vspan}{\text{span}}
\DeclareMathOperator*{\rank}{\text{rank}}

\makeatletter
\let\@fnsymbol\@arabic
\makeatother

\usepackage{breakurl}
\usepackage{url}

\newtheorem{theo}{Theorem}[section]

\newtheorem{defi}[theo]{Definition}
\newtheorem{kor}[theo]{Corollary}
\newtheorem{prop}[theo]{Proposition}
\newtheorem{lem}[theo]{Lemma}
\theoremstyle{definition}

\newtheorem{ass}{Assumption}
\newtheorem{eks}[theo]{Example}

\hypersetup{pdftitle={Existence of a Unique Quasi-stationary Distribution for Stochastic Reaction Networks}}
\hypersetup{pdfauthor={Mads Christian Hansen and Carsten Wiuf}}

\title{Existence of a Unique Quasi-stationary Distribution for Stochastic Reaction Networks}

\author{Mads Christian Hansen\footnote{Dept. of Math. Sciences, Univ. of Copenhagen, Universitetsparken 5, 2100 Copenhagen, Denmark.}\,\,$^,$\footnote{Corresponding author: mads@math.ku.dk} \and Carsten Wiuf$^1$}

\begin{document}

\maketitle

\begin{abstract}
  In the setting of stochastic dynamical systems that eventually go extinct, the quasi-stationary distributions are useful to understand the long-term behavior of a system before evanescence. For a broad class of applicable continuous-time Markov processes on countably infinite state spaces, known as reaction networks, we introduce the inferred
notion of absorbing and endorsed sets, and obtain sufficient conditions for the existence and uniqueness of a quasi-stationary distribution within each such endorsed set. In particular, we obtain sufficient conditions for the existence of a globally attracting quasi-stationary distribution in the space of probability measures on the set of endorsed states. Furthermore, under these conditions, the convergence from any initial distribution to the quasi-stationary distribution is exponential in the total variation norm.
\end{abstract}

\section{Introduction}
We may think of reaction networks in generality as a natural framework for representing systems of transformational interactions of entities \cite{systems}. The set of entities (species) may in principle be of any nature, and specifying not just which ones interact (stoichiometry and reactions) but also quantifying how frequent they interact (kinetics), we obtain the dynamical system of a reaction network. Examples abound in biochemistry, where the language originated, however the true power of this approach is the ability to model diverse processes such as found in biological \cite{stocoli,bio2}, medical \cite{cancer}, social  \cite{political}, computational \cite{comps}, economical \cite{game}, ecological \cite{pak} or epidemiological \cite{nasal} contexts.

\medskip
Whether the universe is inherently deterministic or stochastic in nature, the lack of complete information in complex systems inevitably introduces some degree of stochasticity. Thus, a stochastic description is not an alternative to the deterministic approach, but a more complete one \cite{nonstoch}. Indeed, the deterministic model solution is an approximation of the solution for the stochastic model, improving with the system size, and in general only remaining valid on finite time intervals \cite{kurtz}. Thus, the long-term behavior of a given reaction network may depend crucially on whether it is modeled deterministically or stochastically \cite{scale}. In particular, the possibility of extinction, which is a widely occurring phenomenon in nature, may sometimes only be captured by the latter \cite{extinction}. As a consequence, the counterpart to a stable stationary solution in the deterministically modeled system is not generally a stationary distribution of the corresponding stochastic model. Instead, a so-called quasi-stationary distribution, which is a stationary measure when conditioned on the process not going extinct, has shown to be the natural object of study. A concise overview of the history and current state of this field can be found in \cite{JOR}, while \cite{biblio} contains a comprehensive bibliography on quasi-stationary distributions and related work.
\\\\
From a modeling standpoint, when the copy-numbers of interacting entities are low and reaction rates are slow, it is important to recognize that the individual reaction steps occur discretely and are separated by time intervals of random length \cite{stocoli}. This is for example the case at the cellular level \cite{stogene}, where stochastic effects resulting from these small numbers may be physiologically significant \cite{comps}. Furthermore, stochastic variations inherent to the system may in general be beneficial for identifying system parameters \cite{listnoise}. The quasi-stationary distribution possesses several desirable properties in this domain. Most importantly, if the system under study has been running for a long time, and if the only available knowledge about the system is that it has not reached extinction, then we can conclude that the quasi-stationary distribution, if it exists and is unique, is the likely distribution of the state variable \cite{nasal}.

\medskip
Consider a right-continuous time-homogenous Markov process $(X_t\colon t\geq 0)$ \cite{RW}, that evolves in a domain $D\subseteq \mathbb{R}^d$, wherein there is a set of absorbing states, a ``trap", $A\subset D$. The process is absorbed, also referred to as being killed, when it hits the set of absorbing states, implying $X_t\in A$ for all $t\geq \tau_A$, where $\tau_A=\inf\{t\geq0: X_t\in A\}$ is the hitting time of $A$. As we are interested in the process before reaching $A$, there is no loss of generality in assuming $X_t=X_{t\wedge \tau_{A}}$. We refer to the complement,
\begin{align*}
E:=D\backslash A,
\end{align*}
as the set of endorsed states. For any probability distribution, $\mu$, on $E$, we let $\mathbb{P}_\mu$ and $\mathbb{E}_\mu$ be the probability and expectation respectively, associated with the process $(X_t\colon t\geq 0)$, initially distributed with respect to $\mu$. For any $x\in E$, we let $\mathbb{P}_x=\mathbb{P}_{\delta_{x}}$ and $\mathbb{E}_x=\mathbb{E}_{\delta_x}$. Under suitable conditions, the process hits the absorbing set almost surely (a.s.), that is $\mathbb{P}_x(\tau_{A}<\infty)=1$ for all $x\in E$, and we investigate the behavior of the process before being absorbed \cite{QSD}.

\begin{defi}
A probability measure $\nu$ on $E$ is called a quasi-stationary distribution (QSD) for the process $(X_t\colon t\geq 0)$ absorbed at $A$, if for every measurable set $B\subseteq E$
\begin{align*}
\mathbb{P}_\nu(X_t\in B\,|\, t<\tau_A)=\nu(B), \qquad t\geq 0,
\end{align*}
or equivalently, if there exists a probability measure $\mu$ on $E$ such that
\begin{align*}
\lim_{t\to \infty}\mathbb{P}_\mu(X_t\in B\,|\, t<\tau_A)=\nu(B),
\end{align*}
in which case we also say that $\nu$ is a quasi-limiting distribution.
\end{defi}

We refer to \cite{pop} for a proof of the equivalence of quasi-limiting and quasi-stationary distributions. Existence and uniqueness of a QSD on a finite state space is well known \cite[Chapter 3]{QSD}, and it is given by the normalized left Perron-Frobenius eigenvector of the transition rates matrix restricted to $E$. For the infinite dimensional case, most work has been carried out for birth-death processes in one dimension \cite{JOR}, where classification results yielding information about the set of QSDs exist \cite{VD}.

\medskip
In the present paper, we will focus on a special case of multidimensional processes on countable infinite state spaces which can be viewed as reaction networks. We will prove as the main result in Theorem \ref{res} and Corollary \ref{koret} sufficient conditions for the existence of a unique globally attracting QSD in the space of probability distributions on $E$, equipped with the total variation norm, $\|\cdot\|_{TV}$. Recall that this norm may be defined as \cite{totalvar}
\begin{align*}
\|\mu\|_{TV}=2\sup_{B\subseteq E}|\mu(B)|.
\end{align*}
Thus, informally, the metric associated to this norm is the largest possible difference between the probabilities that two probability distributions can assign to the same event. Our result is based on the following recent result \cite[Theorem 2.1]{Expo}.

\begin{theo}
The following are equivalent
\begin{itemize}
  \item There exists a probability measure $\nu$ on $E$ and two constants $C,\gamma>0$ such that, for all initial distributions $\mu$ on $E$,
  \begin{align*}
\|\mathbb{P}_\mu(X_t\in \cdot \,|\, t<\tau_A)-\nu(\cdot)\|_{TV}\leq Ce^{-\gamma t}, \qquad \forall t\geq 0.
\end{align*}
  \item There exists a probability measure $\nu$ on $E$ such that
  \begin{enumerate}
  \item[(A1)] there exists $t_0, c_1>0$ such that for all $x\in E$,
  \begin{align*}
\mathbb{P}_x(X_{t_0}\in \cdot\,|\, t_0<\tau_A)\geq c_1 \nu(\cdot),
\end{align*}
  \item[(A2)] there exists $c_2>0$ such that for all $x\in E$ and $t\geq 0$,
  \begin{align*}
\mathbb{P}_\nu(t<\tau_A)\geq c_2 \mathbb{P}_x(t<\tau_A).
\end{align*}
\end{enumerate}
\end{itemize}
\end{theo}

Now, using Foster-Lyapunov theory \cite{mt, meyn}, a series of assumptions on the process $(X_t\colon t\geq 0)$ has been shown to be sufficient for (A1) and (A2) to hold \cite{cd}. This approach has been applied to a particular case of multidimensional birth-death processes, giving sufficient conditions, in terms of the parameters of the process, for the existence and uniqueness of a QSD. Here, we extend this result, not just to a larger set of parameter values in the birth-death process case, but to the \textit{much} broader class of stochastic processes known as stochastic reaction networks.

\medskip
The outset of the paper is as follows. In section 2, we introduce the setup and notation of reaction network theory, and define the central inferred notions of endorsed and absorbing states for this class of processes. Section 3 contains the terminology and main assumptions that we shall use throughout the paper. We then move on in section 4, to prove that the processes associated with stochastic reaction networks do indeed satisfy all the required assumptions made by \cite[Corollary 2.8]{cd}. Section 5 contains the main result, Theorem \ref{res}. Finally, we give some examples in section 6, illustrating the applicability of the results.

\section{Reaction Network Setup}

Denote the real numbers by $\mathbb{R}$, the integers by $\mathbb{Z}$, the natural numbers by $\mathbb{N}=\{1,2,\dots\}$ and the nonnegative integers by $\mathbb{N}_0=\{0,1,2,\dots\}$. Further, for any set, $B$, let $|B|$ denote its cardinality and denote by $\mathbbm{1}_B\colon D\to \{0,1\}$ the indicator function of a subset $B\subseteq D$.

\medskip
A reaction network is a triple $\mathcal{N}=(\mathcal{S}, \mathcal{C}, \mathcal{R})$, where $\mathcal{S}$ is a  finite ordered set of species\footnote{The terminology ``species" is standard, although one may equally think of them as general entities or agents.}, $\mathcal{C}$ is a  finite set of complexes, consisting of linear combinations over $\mathbb{N}_0$ of the species, and  $\mathcal{R}\subset \mathcal{C}\times \mathcal{C}$ is an irreflexive relation on $\mathcal{C}$, referred to as the set of reactions \cite{stochreac,MF,crnt}. Furthermore, $\mathcal{R}$ is assumed to be ordered.

\medskip
We define the dimension of the reaction network, $d=|\mathcal{S}|$. Any species $S_i\in \mathcal{S}$ can be identified with the unit vector $e_i\in \mathbb{N}_0^{d}$, thus any complex $y\in \mathcal{C}$ can be identified with a vector in $\mathbb{N}_0^{d}$. It is customary to denote an element $(y_k,y_k')\in \mathcal{R}$ by $y_k\to y_k'\in \mathcal{R}$ in which case we refer to $y_k$ as the source complex and to $y_k'$ as the product complex of reaction $k$. We may thus write $\mathcal{R}=\{y_k\to y'_k\colon  k=1,\dots,r\}$.  Employing a standard, although slight abuse of, notation, we identify $\mathcal{S}=\{S_1,\dots,S_d\}$ with the set $\{1,\dots, d\}$ and $\mathcal{R}$ with $\{1,\dots,r\}$.
We write the $k$'th reaction with the notation
\begin{align*}
\sum_{i\in \mathcal{S}}y_{ki}S_i\to \sum_{i\in \mathcal{S}}y_{ki}'S_i,
\end{align*}
where $y_{ki}=(y_k)_i$ and $y'_{ki}=(y'_k)_i$ are the stoichiometric coefficients associated with the source and product complexes of reaction $k$, respectively. Define the reaction vectors $\xi_k=y'_k-y_k$ and the stoichiometric matrix
\begin{align*}
\Xi=(\xi_1\, \xi_2\, \dots\, \xi_r)\in \mathbb{N}_0^{d\times r}.
\end{align*}
The order of reaction $k$ is the sum of the stoichiometric coefficients of the source complex, $\sum_{i\in \mathcal{S}} y_{ki}$. Finally, we define the maximum of a vector over the set $\mathcal{R}$, $x=\max_{k\in\mathcal{R}}y_k$, as the entry-wise maximum, $x_i=\max_{k\in \mathcal{R}}y_{ki}$. 

\medskip
 A set of reactions $\mathcal R$ induces a set of complexes and a set of species, namely the complexes and species that appear in the reactions. We will assume that a reaction network is always given in this way by $\mathcal R$, and one may then completely describe a reaction network in terms of its reaction graph, whose nodes are the complexes and whose directed edges are  the  reactions. This concise description will be employed in the rest of the paper. 
To avoid trivialities, we assume $\mathcal R\not=\emptyset$.

\medskip
For each reaction we specify an intensity function $\lambda_k\colon\mathbb{N}_0^d\to[0,\infty)$, $k\in\mathcal R$, which satisfies the stoichiometric admissibility condition:
\begin{align*}
 \lambda_k(x) > 0 \quad \Leftrightarrow \quad x\geq  y_{k},
\end{align*}
\noindent
where we use the usual vector inequality notation; $x\geq y$ if $x_i\geq y_i$ for all $i\in\mathcal{S}$. Thus, reactions are only allowed to take place whenever the copy-numbers of each species in the current state is at least as great as those of the corresponding source complex. 
A widely used example is stochastic mass action kinetics given by
\begin{align*}
\lambda_k(x)=\alpha_k\prod_{i=1}^dy_{ki}!\binom{x}{y_k}=\alpha_k\prod_{i=1}^d\frac{x_i!}{(x_i-y_{ki})!},
\end{align*}
for some reaction rate constants $\alpha_k>0$ \cite{stochreac}. The idea is that the rate is proportional to the number of distinct subsets of the molecules present that can form the input of the reaction. It reflects the assumption that the system is well-stirred \cite{stochreac}. Other examples include power law kinetics or generalized mass action kinetics \cite{global,horn,generalized}. A particular choice of such rate functions constitute a stochastic kinetics $\lambda=(\lambda_1,\dots, \lambda_r)$ for the reaction network $\mathcal{N}$, and the pair $(\mathcal{N},\lambda)$ is referred to as a stochastic reaction system, or simply a reaction network with kinetics $\lambda$.

\medskip
We may then specify the stochastic process $(X_t\colon t\geq 0)$ on the state space $D:=\mathbb{N}_0^d$ related to the reaction system $(\mathcal{N},\lambda)$. Let $X_t$ be the vector in $\mathbb{N}_0^d$ whose entries are the species counts at time $t$. If reaction $y_k\to y'_k$ occurs at time $t$, then the new state is $X_t=X_{t-}+y'_k-y_k=X_{t-}+\xi_k$, where $X_{t-}$ denotes the previous state. The stochastic process then follows,
\begin{align}\label{poisson}
X_t=X_0+\sum_{k\in \mathcal{R}}Y_k\left(\int_0^t \lambda_k(X_s)\,ds\right)\xi_k,
\end{align}
where $Y_k$ are independent and identically distributed unit-rate Poisson processes \cite{stochreac,EK,markov}. This stochastic equation is referred to as a random time change representation.  We assume throughout the paper that the process is non-explosive, so that the process is well defined. Assumption \ref{ass1}, though, will imply non-explosiveness.

\subsection{The State Space}

To define the set of endorsed states and absorbing states in the setting of stochastic reaction networks, we recall some terminology from stochastic processes. We say that there is a path from $x$ to $y$, denoted $x\mapsto y$, if there exists $t\geq0$ such that $\mathbb{P}_x(X_t=y)>0$. We extend this notion to sets as follows; $B_1\mapsto B_2$ if there exists $x\in B_1$ and $y\in B_2$ such that $x\mapsto y$. Finally, we introduce the region of large copy numbers, where all reactions may take place, defined as
\begin{align*}
R=\{x\in \mathbb{N}_0^d\,|\, \lambda_k(x)>0 \,\forall\, k\in \mathcal{R}\}.
\end{align*}
Any network satisfies $R\neq \emptyset$. Indeed, by the stoichiometric compatibility condition, $\{x\in D\,|\, x\geq M\}\subseteq R$ where $M=\max_{k\in \mathcal{R}} y_k\in \mathbb{N}_0^d$.
Letting $D_E=\{x\in D\,|\, x\mapsto R\}$, we may decompose the state space into a disjoint union
\begin{align*}
D=D_E\sqcup D_A.
\end{align*}
A state space $D$ is irreducible if for all $x,y\in D$ we have $\mathbb{P}_x(X_{t_1}=y)>0$ and $\mathbb{P}_y(X_{t_2}=x)>0$ for some $t_1,t_2>0$ \cite{scale}. Thus, $D$ is irreducible if for all $x,y\in D$  there exists a path $x\mapsto y$. Irreducibility induces a class structure on the state space \cite{markov}, and we denote the classes by $\mathcal{I}_1,\mathcal{I}_2,\dots$ (potentially infinitely many). Let $\mathscr{I}$ denote the set of irreducible classes.  Obviously, either $\mathcal I_i\subseteq D_E$ or $\mathcal I_i\subseteq D_A$, $i\ge 1$.

\begin{lem}
The pair $(\mathscr{I}, \preceq)$, where $\preceq$ is given by
\begin{align*}
\mathcal{I}_j\preceq \mathcal{I}_i \Leftrightarrow \mathcal{I}_i\mapsto  \mathcal{I}_j,\qquad i,j\ge 1,
\end{align*}
is a well defined poset. The irreflexive kernel $(\mathscr{I},\prec)$ gives a well defined strict poset. 
\end{lem}
\begin{proof}
Since all elements $\mathcal{I}\in \mathscr{I}$ are irreducible, there exists a path between any two points in $\mathcal{I}$ hence $\mathcal{I}\preceq \mathcal{I}$ yielding the relation reflexive.

Suppose $\mathcal{I}_i\preceq \mathcal{I}_j$ and $\mathcal{I}_j\preceq \mathcal{I}_i$ for some $i,j\ge 1$. Let $x\in \mathcal{I}_i$ and $y\in \mathcal{I}_j$ be given. By assumption, we may find a path from $x$ to some $z_2\in \mathcal{I}_j$, and by irreducibility of $\mathcal{I}_j$ there is a path from $z_2$ to $y$. Similarly, we may by assumption find a path from $y$ to some $z_1\in \mathcal{I}_i$ and by irreducibility of $\mathcal{I}_i$ a path from $z_1$ to $x$. As $x,y$ were arbitrary, we conclude that there exists a path between any two points in $\mathcal{I}_i\cup \mathcal{I}_j$ hence $\mathcal{I}_i=\mathcal{I}_j$, yielding the relation antisymmetric.

Finally, suppose $\mathcal{I}_k\preceq \mathcal{I}_j$ and $\mathcal{I}_j\preceq \mathcal{I}_i$ for some $i,j,k\ge 1$. Then there exists a path from some $x\in \mathcal{I}_i$ to some $z_1\in \mathcal{I}_j$ and a path from some $z_2\in \mathcal{I}_j$ to some $y\in \mathcal{I}_k$. By irreducibility of $\mathcal{I}_j$ there is a path from $z_1$ to $z_2$, and concatenation of the three paths yield one from $x$ to $y$. We conclude that $\mathcal{I}_k\preceq \mathcal{I}_i$, hence the relation is transitive.
\end{proof}

A similar ordering has been considered in \cite{DP}. However, their further analysis rests on the setting of discrete time, rendering the approach insufficient for stochastic reaction networks. To exploit the graphical structure induced by $\preceq$, define the marked directed acyclic graph $\mathcal{D}=(\mathscr{I}, \mathscr{E})$ as follows. The set of directed edges is
\begin{align*}
\mathscr{E}=\{(\mathcal{I}_i,\mathcal{I}_j)\in \mathscr{I}^2\,|\, \mathcal{I}_j\prec \mathcal{I}_i, i,j\ge 1\},
\end{align*}
while the marking $\mathscr{I}=\mathscr{I}_A\sqcup \mathscr{I}_E$ is given by
\begin{align*}
\mathscr{I}_E=\{\mathcal{I}\in \mathscr{I}\,|\, \mathcal{I} \subseteq D_E\}, \qquad \mathscr{I}_A=\{\mathcal{I}\in \mathscr{I}\,|\, \mathcal{I} \subseteq D_A\}.
\end{align*}
Let $\mathscr{V}_1,\mathscr{V}_2,\dots$ denote the vertex set of the respective connected components of the induced subgraph $\mathcal{D}[\mathscr{I}_E]$, the graph with vertex set $\mathscr{I}_E$ and edges from $\mathscr{E}$ with start and end nodes in $\mathscr{I}_E$.
\begin{defi}\label{statespace}
The endorsed sets and absorbing sets are defined, respectively, by \begin{align*}
E_n=\bigcup_{\mathcal{I}\in \mathscr{V}_n} \mathcal{I},
\qquad
A_n=\bigcup_{\substack{\mathcal{I}\in \mathscr{I}_A\colon E_n\mapsto \mathcal{I}}}\mathcal{I}, \qquad n\geq 1.
\end{align*}
The corresponding state space is defined by $D_n=E_n\sqcup A_n$.
\end{defi}

\vspace{-0.2cm}
\begin{figure*}[h!]
\begin{multicols*}{2}
\begin{displaymath}
    \xymatrix@C-=0.4cm@R-=0.4cm{& 2S_1\\
    S_1+2S_2 \ar[rd]\ar[ru]& \\
    & 3S_2}
\end{displaymath}
\columnbreak
\begin{center}
 \def\svgwidth{130pt}
  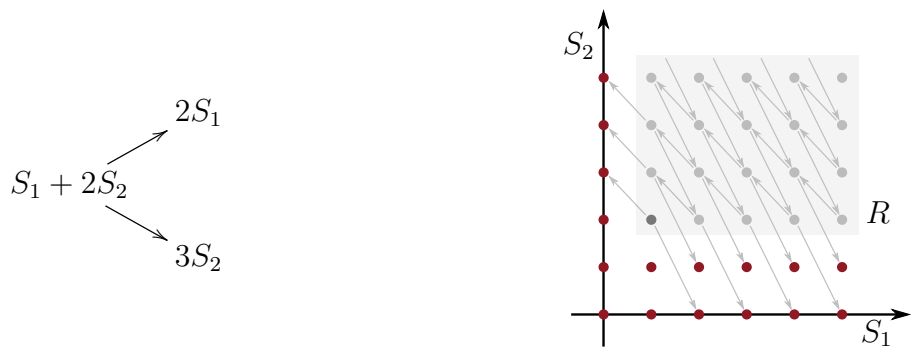
\end{center}
\end{multicols*}
\vspace{-0.55cm}
\caption{Left: The reaction graph of a stochastic reaction network. Right: The state space with the region $R$ in shaded grey. Points in $D_A$ are marked red.}
\label{counter}
\end{figure*}
\newpage
As $D_E$ is non-empty, the existence of at least one endorsed set is guaranteed. By construction, the endorsed sets are disjoint, their union is $D_E$ and their number, $N_E$, may in general be countable infinite. Furthermore, any absorbing set is confined to a subset of $\{x\in \mathbb{N}_0^d\,|\, x\ngeq M\}$,  lying ``close'' to the boundary of $\mathbb{N}_0^d$. If extinction is possible from an endorsed class $E_n$ then this absorption will take place in $A_n$. Note, however, that the set $\mathscr{I}_A$ may in general be empty, in which case no absorbing set exist. To further illuminate the structure of the endorsed sets, we provide the following classification result.
\begin{prop}\label{state}
For $M\in \mathbb{N}^d$ sufficiently large, the set $\{x\in D_E\,|\,x\geq M\}$ intersects
\begin{enumerate}
 \item[(i)] finitely many endorsed sets if and only if $\rank\,\Xi=d$.
  \item[(ii)] a single endorsed set if and only if $\vspan_{\mathbb{Z}}\Xi=\mathbb{Z}^d$.
\end{enumerate}
\end{prop}
\begin{proof}
Suppose first that $\rank\, \Xi<d$. Let $x\in D_E$. Then $x\in E_1$, say and there exists a $y\in D_E$ such that $y\notin (x+\vspan_{\mathbb{R}}\Xi)$. In particular, $y\in E_2$ where $E_1\neq E_2$. This procedure can be repeated indefinitely, yielding infinitely many endorsed sets.

Now, suppose $\rank\,\Xi=d$. Then one may choose a linear combination of the reaction vectors yielding a strictly positive point,
\begin{align*}
\sum_{k\in \mathcal{R}}a_k\xi_k>0, \qquad a_k\in \mathbb{Z}.
\end{align*}
Let $a=\sum_{k\in \mathcal{R}}|a_k|$ and $a_0=0$. Define the sequence $(w_\ell)_{\ell=1,\dots,a}$ by
\begin{align*}
w_\ell=\xi_k, \qquad \ell=1+\sum_{j=0}^{k-1}|a_j|,\dots, \sum_{j=1}^{k}|a_{j}|, \qquad k\in \mathcal{R}.
\end{align*}
As the reaction vectors are finite, the partial sums $P_{j}=\sum_{\ell=1}^jw_\ell$ are finite for each $j\leq a$. Let
\begin{align*}
m=\min_{i\in \mathcal{S}, j\leq a}(P_j)_i.
\end{align*}
Choosing each coordinate $M_i>|m|+\max_{k\in \mathcal{R}}y_k$ for each $i\in \mathcal{S}$, it follows that any point $x\in D_E$ with $x\geq M$ satisfies $x+P_j\in R$ for any $j\leq a$. We say that a sequence of states $(x_1,\dots,x_n)$ is an undirected walk from $x_1$ to $x_n$ if for all $1\leq i\leq n-1$ there exists $k(i)\in \mathcal{R}$ such that $x_{i+1}=x_{i}\pm \xi_{k(i)}$. As all reactions may occur in $R$, we conclude that $x$ has an undirected walk to the point
\begin{align*}
x':=x+P_a=x+\sum_{k\in \mathcal{R}} a_k \xi_k>x.
\end{align*}
Thus, by definition, $x$ and $x'$ belong to the same endorsed set, say $E_1$. To determine the number of endorsed sets, let $B=(b_i)$ denote the basis matrix of the free $\mathbb{Z}$-module generated by the stoichiometric matrix $\Xi$, and define the lattice generated by $\Xi$ to be
\begin{align*}
\mathcal{L}(\Xi)=\left\{\sum_{i=1}^{\rank \Xi}z_ib_i\colon z_i\in \mathbb{Z}\right\}.
\end{align*}
By the construction in the previous paragraph, it follows that all points in the region $\{y\in E_1\colon y\geq M\}$ belong to the same translated lattice $x+\mathcal{L}(\Xi)$. By assumption, $\vspan_{\mathbb{R}}\Xi =\mathbb{R}^d$ hence the lattice $\mathcal{L}(\Xi)$ has rank $d$. The number of ways one may translate a rank $d$ lattice in $\mathbb{R}^d$ to an integer lattice point without any points intersecting is given by considering the number of integer lattice points inside the fundamental parallelotope,
\begin{align*}\textstyle
\mathcal{P}(\Xi)=\left\{\sum_{i=1}^d\theta_ib_i \,|\, \theta_i\in [0,1), b_i\in \vspan_{\mathbb{Z}}\Xi, \det(B)\neq 0 \right\}.
\end{align*}
Indeed, as $\mathcal{P}(\Xi)$ tiles $\mathbb{R}^d$, that is for any point $z\in \mathbb{R}^d$ there exists a unique $z'\in \mathcal{L}(\Xi)$ such that $z\in z'+\mathcal{P}(\Xi)$, the problem is reduced to a single fundamental parallelotope, which by definition contains exactly one point from each translated lattice \cite{lattice}. Further, the number of integer lattice points inside $\mathcal{P}(\Xi)$ is exactly equal to the volume of the parallelotope \cite[p. 97]{para}, hence, by finiteness of the reaction vectors,
\begin{align*}
N_E(M)=|\det(B)|<\infty,
\end{align*}
for $M$ sufficiently large, where $N_E(M)$ is the number of endorsed sets intersecting $\{x\in D_E\,|\, x\geq M\}$. This proves (i) of the proposition.

 Finally, if $\vspan_{\mathbb{Z}}\Xi=\mathbb{Z}^d$ then $\mathcal{L}(\Xi)=\mathbb{Z}^d$ and the unit vectors $e_1,\dots, e_d\in \mathcal{L}(\Xi)$. As $\mathcal{P}(I_d)\cap \mathcal{L}(\Xi)=\{0\}$ we conclude from \cite{lattice} that $e_1,\dots, e_d$ is a basis for $\mathcal{L}(\Xi)$ hence $N_E(M)=1$ as desired.
\end{proof}

Note that, in particular, a reaction network whose associated stochastic process is a birth-death process, that is, a process where for each $i=1,\dots, d$ either $e_i\in \mathcal{R}$ or $-e_i\in \mathcal{R}$, have a single endorsed set for $x$ sufficiently large. In practice, one may find the endorsed sets by picking $x\in R$ and adding states by a backtracking algorithm \cite{craciun}. Verification of $\vspan_{\mathbb{Z}}\Xi=\mathbb{Z}^d$ can be done by calculation of the Hermite normal form \cite{craciun}.

\medskip
One may suspect that Proposition \ref{state} could be strengthened to hold on the entire set $D_E$. This is only partially true. Consider as an example the three-dimensional reaction network given by the reaction graph
\begin{displaymath}
    \xymatrix@C-=0.4cm@R-=0.4cm{ & 3S_1\\
    S_1+S_2+S_3 \ar[ru] \ar[r]\ar[rd]& 2S_1\\
    & 2S_3}
    \vspace{0.3cm}
\end{displaymath}
\noindent
It follows that $\rank{\Xi}=3$. However, $D_E=\mathbb{N}^3$, $D_A=\mathbb N_0^3\setminus \mathbb N^3$, and each singleton $\{(1,m,1)\}$, $m\geq 1$, constitutes its own endorsed set. Thus, in the generic picture, close to the absorbing set, there may be infinitely many endorsed sets. We do, however, have the following corollary.
\begin{kor}\label{dim2}
If $d\leq 2$, there are finitely many endorsed sets if and only if $\rank\,\Xi=d$.
\end{kor}
\begin{proof}
We only need to prove that if $\rank\Xi=d$ then there are finitely many endorsed sets. For this, it suffices to prove that at most finitely many $x\in D_E$ do not have an undirected walk (as introduced in the proof of Proposition \ref{state}) to a point $z\geq M$. Indeed, by the proof of Proposition \ref{state}, if such a path exists, then by definition $x$ belongs to one of finitely many endorsed sets. With the remaining set being finite, the total number of endorsed sets is therefore finite.

Now, let $x\in D_E$ be given. By the definition of endorsed sets, there exists a path $x\mapsto y$ with $\lambda_k(y)>0$ for all $k\in \mathcal{R}$. As $\rank\Xi=d$, for each $1\leq j\leq d$, there exists $k(j)\in \mathcal{R}$ such that $\langle e_j,\xi_{k(j)}\rangle\neq0$. Keeping the $j$th coordinate fixed and increasing the possible other if necessary, thus arriving at a point $y'$ with $y_j=y_j'$ and $y_i\leq y_i'$ for $i\neq j$, by the stoichiometric compatibility condition we may by repeated use of reaction $k(j)$ find a path $y'\mapsto z$ or $z\mapsto y'$ with $z\geq M$. Repeating the argument for the possible remaining coordinate, we conclude that there exists an $M'\in \mathbb{N}^d$ such that if $x\nless M'$ then $x$ has an undirected walk to a point $z\geq M$. As the set $\{x\in D_E\,|\, x<M'\}$ is finite, this concludes the proof.
\end{proof}

One may easily verify, that the second part of Proposition \ref{state} can also be extended for $d=1$. Indeed, for any point $x\in D_E$ there exists a reaction $k\in \mathcal{R}$ such that $\lambda_k(x)>0$ and either $x+\xi_k>x$ or $x+\xi_k<x$. Otherwise $x\in D_A$. Consequently, there is a point $z\ge M$ for any $M\in D_E$ such that either $x\mapsto z$ or $z\mapsto x$. However, note that the network in Figure \ref{counter} shows that this result does not hold in the case $d=2$.

\medskip
An endorsed set, $E_n$, $n\ge 1$, is only irreducible if it consists of a single irreducible class. If there is more than one irreducible class in $E_n$, then we need that there is a smallest one to ensure uniqueness of a QSD.

\begin{ass}\label{class}
For a given endorsed class $E_n$, $n\ge 1$, we assume:
\begin{enumerate}
  \item[(i)] $E_n$ contains a unique minimal irreducible class, $\mathcal I^n_{\min}$.
  \item[(ii)] if $A_n\neq \emptyset$ then $\mathcal I^n_{\min}\mapsto A_n$.
\end{enumerate}
\end{ass}

We shall see that Assumption \ref{class}(i) is equivalent to a more technical property of the state space, which is necessary for our results to hold. Thus no generality is lost in having Assumption \ref{class}(i). 

\medskip
Networks without any minimal class exists, for example $\emptyset\to S_1$, which does not have an absorbing set either. Furthermore, networks with more than one minimal class also exist, for example, $S_1+S_2\to \emptyset$, $S_2\to \emptyset$. Thus Assumption \ref{class}(i) is indeed not superfluous. We believe Assumption \ref{class}(ii) is always met if Assumption \ref{class}(i) is. It ensures that one may always reach the absorbing set, if it is non-empty.

\medskip
Definition \ref{statespace} accommodates the general case where uniqueness of a QSD does not necessarily hold, in which case the support of the QSD may stretch the entire endorsed set rather than, as we shall see, the unique minimal irreducible class. We remark that rather than investigating an entire endorsed set, one may be interested in a particular irreducible component, say $\mathcal{I}$. Letting the state space be $D=E\sqcup A$ where
\begin{align*}
E=\mathcal{I}, \qquad  A=\bigcup_{\mathcal{J}\in \mathscr{I}_A\colon E\mapsto \mathcal{J}}\mathcal{J},
\end{align*}
the theory to be developed in this paper applies to this case as well.

\medskip
As an illuminating example consider the generalized death process $mS_1\to \emptyset$ with $m\in \mathbb{N}$, where each point in the state space $D=\mathbb{N}_0$ constitutes its own irreducible class. Here, the endorsed and absorbing sets are $E_n=\{n+pm-1 \,|\, p\in \mathbb{N}\}$ and $A_n=\{n-1\}$ respectively, for $n=1,\dots,m$, and Assumption \ref{class} is satisfied for all $n$. Thus, $D_E=\{m,m+1,\dots\}$ and $D_A=\{0,\dots,m-1\}$. It is known that in the simple death case, $m=1$, uniqueness does not hold on $D_E=\mathbb{N}$. Indeed, there is a continuum of QSDs with support larger than $\{1\}$, the unique minimal class \cite{AG}.
\\\\
\begin{figure}[h!]
\begin{center}
 \def\svgwidth{140pt}
  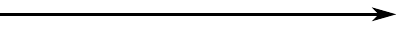
\end{center}
\vspace{-0.3cm}
\caption{State space of the reaction network $2S_1\to \emptyset$. There are two endorsed sets.}
\end{figure}
\vspace{-0.1cm}

\medskip
In the setting of birth-death processes in one dimension, which has an infinite state space, it is known that there will be either none, a unique or a continuum of QSDs \cite{VD}. Consider the two reaction networks
\begin{align}\label{1d}
\emptyset \overset{\raise0.2em\hbox{\scriptsize$\alpha_1$}}{\leftarrow}S_1\overset{\raise0.2em\hbox{\scriptsize$\alpha_2$}}{\rightarrow} 2S_1, \qquad\qquad \emptyset  \overset{\raise0.2em\hbox{\scriptsize$\alpha_1$}}{\leftarrow} S_1\underset{\raise0.2em\hbox{\scriptsize$\alpha_3$}}{\overset{\raise0em\hbox{\scriptsize$\alpha_2$}}{\rightleftharpoons}} 2S_1,
\end{align}
endowed with mass action kinetics. In both cases we conclude, according to Definition \ref{statespace}, that the set of absorbing states and the set of endorsed states are
\begin{align*}
D_A=A_1=\{0\}, \qquad D_E=E_1=\{1,2,\dots\},
\end{align*}
respectively. For the network on the left in (\ref{1d}), assuming $\alpha_1>\alpha_2$, there is a continuum of QSDs on $E_1$, while for the network on the right, there is a unique QSD on $E_1$, for all parameter values \cite{pop}. This fits well with our result -- the necessity of having reactions of order higher than one to ensure uniqueness permeates to higher dimensions.

\section{Extension of Arguments}

In this and the following sections, we shall simply use the notation $E$ to refer to a single endorsed set, with corresponding non-empty absorbing set $A$, when there is no ambiguity. Further, as existence and uniqueness is known on finite state spaces, we shall assume without loss of generality that $E$ is countably infinite. We make the following definitions inspired by \cite{cd} and \cite{mt}.
\begin{defi}\label{normlike}
For any vector $v\in \mathbb{N}^d$, we define a corresponding function $\langle v,\cdot\rangle\colon\mathbb{Z}^d\to \mathbb{Z}$ given by the standard inner product
\begin{align*}
\langle v,x\rangle=x\cdot v.
\end{align*}
\end{defi}

This function may in general take negative values, however, when restricting $\langle v,\cdot\rangle$ to $E$, one obtains a norm-like function, \cite{mt}. Choosing $v=(1,\dots,1)$ we recover the function used in \cite{cd}. In general, we shall choose $v\in \mathbb{N}^d$ based on the particular reaction network at hand, and will in the following consider it fixed. For $n\in \mathbb{N}$, define the sets
 \begin{align*}
O_n=\{x\in E\colon \langle v,x\rangle\leq n\}.
\end{align*}
which, irrespectively of $v$, are compact subsets of $E$, satisfying $O_n\subseteq O_{n+1}$ and $E=\bigcup_{n\in \mathbb{N}}O_n$. We denote the first hitting time of $A$, the first hitting time of $O_n$ and the first exit time of $O_n$ by
 \begin{align*}
 \tau_A=\inf\{t\geq0\colon X_t\in A\},\qquad \tau_{n}=\inf\{t\geq0\colon X_t\in O_n\}, \qquad T_n=\inf\{t\geq0\colon X_t\notin O_n\},
\end{align*}
respectively. Note that all of these are stopping times and might be infinite.  As we will be concerned with the application of unbounded functions serving the purpose of a Lyapunov function, we introduce the weakened  generator, $L$, for the Markov process \cite{mt,Expo}.

\begin{defi}
A measurable function $W\colon D\to \mathbb{R}$ belongs to the domain $\mathcal{D}(L)$ of the weakened  generator $L$ of $(X_t\colon t\geq0)$ if there exists a measurable function $U\colon E\to \mathbb{R}$ such that, for all $n\in \mathbb{N}$, $t\geq0$ and $x\in E$
\begin{align*}
\mathbb{E}_x W(X_{t\wedge T_n})=W(x)+\mathbb{E}_x\left(\int_0^{t\wedge T_n}U(X_s)\,ds\right),
\end{align*}
and
\begin{align}\label{weak2}
\mathbb{E}_x\left|\int_0^{t\wedge T_n}U(X_s)\,ds\right|<\infty,
\end{align}
and we define $LW=U$ on $E$ and $LW\equiv 0$ on $A$.
\end{defi}

As the state space of interest is always countable, all functions $f\colon D\to \mathbb{R}$ are measurable. Moreover,  as the state space is discrete and $O_n$ is finite  for all $n\in \mathbb{N}$, all functions $f\colon D\to \mathbb{R}$ are in the domain of the weakened generator, $\mathcal{D}(L)$ \cite{mt}.  In particular, $\mathbb{E}_x W(X_{t\wedge T_n})$ is well-defined and finite. Generally, the weakened  and the infinitesimal generator need not agree, and the infinitesimal generator may not exist \cite{mt}.

\medskip
However, if $f$ is bounded, then they do agree. In particular,  $E_x|f(X_t)|<\infty$ and it follows that as $t\to 0$,
\begin{align*}
\mathbb{E}_xf(X_t)&=\left(\sum_{k\in \mathcal{R}}f(x+\xi_k)\mathbb{P}_x(X_t=x+\xi_k)\right)+f(x)\mathbb{P}_x(X_t=x)+o(t)\\
&=\left(\sum_{k\in \mathcal{R}}f(x+\xi_k)\lambda_k(x) t+o(t)\right)+f(x)\left(1-\sum_{k\in \mathcal{R}}\lambda_k(x) t+o(t)\right)+o(t)\nonumber\\
&=\sum_{k\in \mathcal{R}}\lambda_k(x)(f(x+\xi_k)-f(x))t+f(x)+o(t).\nonumber
\end{align*}
Hence $\mathbb{E}_xf(X_t)$ is differentiable and from the fundamental theorem of calculus we conclude that the weakened generator coincides with the (weak) infinitesimal generator \cite{mt},
\begin{align*}
\widehat L f(x)=\lim_{t\to 0}\frac{\mathbb{E}_x f(X_t)-f(x)}{t}=\sum_{k\in \mathcal{R}}\lambda_k(x)(f(x+\xi_k)-f(x)),
\end{align*}
for $x\in E$.

\medskip
Moreover,  setting $W(x)=\langle v,x\rangle$ as in Definition \ref{normlike}, it follows from the Poisson characterization of the process \eqref{poisson}, that
\vspace{-0.3cm}
\begin{align*}
\mathbb{E}_x W(X_{t\wedge T_n})&=\mathbb{E}_x \langle v,X_{t\wedge T_n}\rangle
=\mathbb{E}_x \langle v, X_0\rangle+\mathbb{E}_x \langle v,\sum_{k\in \mathcal{R}}Y_k\left(\int_0^{t\wedge T_n} \lambda_k(X_s)\,ds\right)\xi_k\rangle\\
&=W(x)+\mathbb{E}_x\left(\int_0^{t\wedge T_n}\sum_{k\in \mathcal{R}} \lambda_k(X_s)\langle v,\,\xi_k\rangle \,ds\right),
\end{align*}
such that 
\begin{align}\label{p9}
LW(x)=\sum_{k\in \mathcal{R}}\lambda_k(x) \langle v,\xi_k\rangle, \qquad x\in E.
\end{align}
Note that \eqref{weak2} is fulfilled as $O_n$ is finite.

\begin{defi}\label{definitive}
Define the functions $d_v, d^v\colon \mathbb{N}\to \mathbb{R}$ by
\begin{align*}
d_v(n)&=-\max_{x\in E, \langle v,x\rangle=n}\sum_{k=1}^r\lambda_k(x)\langle v,\xi_k\rangle\mathbbm{1}_{E}(x+\xi_k),\\
d^v(n)&=\max_{x\in E, \langle v,x\rangle=n}n\sum_{k=1}^r\lambda_k(x)\mathbbm{1}_{A}(x+\xi_k).
\end{align*}
\end{defi}
All networks, for which extinction is possible, have the property that there exists a $v\in \mathbb{N}^d$ such that $\langle v,\xi_k \rangle\leq0$ for some reaction $k\in \mathcal{R}$. Indeed, suppose that $\langle v,\xi_k\rangle>0$ for all $v\in \mathbb{N}^d$ and $k\in \mathcal{R}$. Then $\xi_k\in \mathbb{N}_0^d$ for all $k\in \mathcal{R}$, hence $x+\xi_k\geq x$ for any $x\in D$. In particular, if $x\in E$ then $x+\xi_k\in E$ and we conclude, from the observation that any absorbing set is confined to a subset of $\{x\in \mathbb{N}_0^d\,|\, x\ngeq M\}$ for some $M$ sufficiently large, that the process is not absorbed. By contraposition the desired claim holds.
Note that for fixed $n\in \mathbb{N}$, the set $\{x\in E, \langle v,x\rangle=n\}$ might be empty, thus we define
\begin{align*}
\mathscr{N}=\{n\in \mathbb{N}\,|\, \exists\, x\in E\colon \langle v,x\rangle=n\},
\end{align*}
and make the following central assumption.
\begin{ass}\label{ass1}
There exists $v\in \mathbb{N}^d$ and $\eta>0$, $N\in \mathbb{N}$, such that, for $n\geq N$,
\begin{align*}
d_v(n)\geq \eta\, d^v(n),
\end{align*}
and, with the limit being taken over $\mathscr{N}\subseteq \mathbb{N}$,
\begin{align*}
\lim_{n\to \infty}\frac{d_v(n)}{n^{1+\eta}}=\infty.
\end{align*}
\end{ass}

We note that, as $d^v(n)$ is always non-negative, this assumption assures that $d_v(n)$ is non-negative for $n$ sufficiently large. We shall see that this assumption further ensures the ability to ``come down from infinity'' in finite time. In the case where there is no absorbing set, the empty sum in Definition \ref{definitive} yields $d^v(n)=0$, and we may reformulate a result of \cite{scale} (see below). In this paper, we will extend the result to the case where $E$ is not necessarily irreducible, but satisfies Assumption \ref{class}(i), and where there may exist a non-empty absorbing set of states (see Theorem \ref{samlet}).

\begin{theo}\label{noabs}
For a reaction network satisfying Assumption \ref{ass1}, with $A=\emptyset$ and $E$ irreducible, the associated stochastic process $(X_t: t\geq 0)$ is exponentially ergodic and thus admits a unique stationary distribution $\pi$. Further, there exist constants $C,\gamma>0$ such that, for all probability measures $\mu$ on $E$,
\begin{align*}
\|\mathbb{P}_\mu(X_t\in \cdot)-\pi(\cdot)\|_{TV}\leq Ce^{-\gamma t}, \qquad t\geq 0.
\end{align*} 
\end{theo}
\begin{proof}
For $x\in D$, let $n=\langle v,x\rangle$. Thus, by Assumption \ref{ass1}(ii), for any constant $c>0$,
\begin{align*}
-\sum_{k=1}^r \lambda_k(x) \langle v,\xi_k\rangle\geq- \max_{x'\in E, \langle v,x'\rangle=n}\sum_{k=1}^r \lambda_k(x') \langle v,\xi_k\rangle=d_v(n)\geq c n^{1+\eta}\geq c \langle v,x\rangle,
\end{align*}
for $n$ larger than some $N\in \mathbb{N}$. The set of $x'\in D$ such that $\langle v,x'\rangle=n\leq N$ is compact, hence there is $c_1>0$ such that
\begin{align*}
d_v(n)\geq c n-c_1=c \langle v,x\rangle-c_1.
\end{align*}
We conclude that for all $x\in D$,
\begin{align*}
\sum_{k=1}^r \lambda_k(x) \langle v,\xi_k\rangle\leq c_1-c\langle v,x \rangle,
\end{align*}
hence from \cite[Proposition 4]{scale} it follows, due to irreducibility of $E$, that there exist constants $C,\gamma>0$ such that for all $x_0\in E$,
\begin{align*}
\|\mathbb{P}_{x_0}(X_t\in \cdot)-\pi(\cdot)\|_{TV}\leq Ce^{-\gamma t}
\end{align*}
for all $t\geq 0$. Finally, if we consider the random starting point $X_0\sim \mu$, we find
\begin{align*}
&\|\mathbb{P}_{\mu}(X_t\in \cdot)-\pi(\cdot)\|_{TV}=\left\|\sum_{x_0\in E}\mu(x_0)(\mathbb{P}_{x_0}(X_t\in \cdot)-\pi(\cdot))\right\|_{TV}\\
&\leq \sum_{x_0\in E}\mu(x_0)\|\mathbb{P}_{x_0}(X_t\in \cdot)-\pi(\cdot)\|_{TV}\leq \sum_{x_0 \in E}\mu(x_0)Ce^{-\gamma t}=Ce^{-\gamma t},
\end{align*}
as required.
\end{proof}

It is sufficient to have $\eta=0$  for Theorem \ref{noabs} to hold. However, as we shall see, if $A\neq \emptyset$ then $\eta>0$ is required. The intuitive meaning is that the quasi-stationary distribution exists on the long-time, but not infinite time horizon, where the process will be absorbed. Thus if the process does not ``come down from infinity in finite time'', that is if $\eta=0$, 
starting close to $A$ will almost surely result in absorption while starting at ``infinity'' will not, contradicting uniqueness of the QSD. When no absorbing set exist, however, the quasi-stationary distribution reduces to the stationary distribution which exists on the infinite time horizon.

\section{Verifying Assumptions}

We start by introducing some notation and definitions from \cite{cd} for ease of reference.
\begin{defi}\label{couple}
A couple $(V,\varphi)$ of measurable functions $V$ and $\varphi$ from $D=E\cup A$ to $\mathbb{R}$ is an admissible couple of functions if
\begin{enumerate}
  \item[(i)] $V$ and $\varphi$ are bounded and nonnegative on $D$, positive on $E$, satisfy $V(x)=\varphi(x)=0$ for all $x\in A$, and further
  \begin{align*}
\inf_{x\in E}\frac{V(x)}{\varphi(x)}>0.
\end{align*}
  \item[(ii)] For all sequences $(x_p)_{p\geq 1}$ in $E$ such that $\{p\in \mathbb{N}\colon x_p\in O_n\}$ is finite for all $n\geq 1$,
\begin{align*}
\lim_{p\to \infty}\frac{V(x_p)}{\varphi(x_p)}=\infty, \qquad \text{and} \qquad
  \lim_{n\to \infty}V(X_{T_n})=0 \quad \mathbb{P}_x\text{-a.s.} \text{ for all $x\in E$}.
\end{align*}
  \item[(iii)] $LV$ is bounded from above and $L\varphi$ is bounded from below.
\end{enumerate}
\end{defi}

The definition of a couple of admissible functions in \cite{cd} further requires that $V$ and $\varphi$ belong to the domain of the weakened infinitesimal generator of $(X_t\colon t\geq 0)$. However, since any function $f\colon\mathbb{N}_0^d\to \mathbb{R}$ is in this domain for discrete state spaces \cite{mt}, the requirement is automatically satisfied. Furthermore, as $V,\varphi$ are bounded, the infinitesimal generator $\widehat L$ is defined hereon and agrees with the weakened generator $L$.
 
\medskip
The question of extinction has recently attracted much attention on its own \cite{ext}. Therefore, we provide the following proposition which renders an explicit criterion for when the stochastic process associated to a stochastic reaction network goes extinct almost surely.  The assumption in the proposition is weaker than Assumption \ref{ass1}.
\begin{prop}\label{lem1}
Under Assumption \ref{class}, with $A\neq \emptyset$,  the process $(X_t\colon t\geq 0)$ is absorbed $\mathbb{P}_x$-a.s. for all $x\in E$ if $d_v(n)>\zeta\frac{d^v(n)}{n}$ for $n$ sufficiently large, where
\begin{align*}
\zeta=\underset{k\in \mathcal{R}_A}{\max}\langle v,\xi_k\rangle, \qquad \mathcal{R}_A=\{k\in \mathcal{R}: (E+\xi_k)\cap A\neq \emptyset\}.
\end{align*}
\end{prop}
\begin{proof}
Define the norm-like function $W(x)=\langle v,x\rangle$ on $\mathbb{N}_0^d$ as in Definition \ref{normlike}, and let $L$ be the weakened infinitesimal generator of $(X_t\colon t\geq 0)$. It follows from (\ref{p9}) and the assumption that $d_v(n)>C\frac{d^v(n)}{n}$ for $n$ sufficiently large that for each $x\in E$ with $\langle v,x\rangle=n$,
\begin{align}
LW(x)=&
\sum_{k=1}^r\lambda_k(x)\langle v,\xi_k\rangle\nonumber\\
=&\sum_{k=1}^r\lambda_k(x)\langle v,\xi_k\rangle  \mathbbm{1}_{E}(x+\xi_k)+\sum_{k=1}^r\lambda_k(x)\langle v,\xi_k\rangle \mathbbm{1}_{A}(x+\xi_k)\nonumber\\
\leq&\max_{x'\in E,\langle v,x'\rangle=n}\sum_{k=1}^r\lambda_k(x')\langle v,\xi_k\rangle  \mathbbm{1}_{E}(x'+\xi_k)+\max_{x'\in E,\langle v,x'\rangle=n}\sum_{k=1}^r\lambda_k(x')\langle v,\xi_k\rangle \mathbbm{1}_{A}(x'+\xi_k)\nonumber\\
\leq& -d_v(n)+\zeta\max_{x'\in E, \langle v,x'\rangle=n}\sum_{k=1}^r \lambda_k(x')\mathbbm{1}_{A}(x'+\xi_k)=-d_v(n)+\frac{\zeta}{n}d^{v}(n)<0,\label{aqua}
\end{align}
for $n$ sufficiently large. In particular, there exists an $N\in \mathbb{N}$ such that for $n=\langle v,x\rangle \geq N$, we have $LW(x)<0$, hence, setting $M=\max_{x\in E\colon1\leq \langle v,x\rangle \leq N}\{0,LW(x)\}$, yields
\begin{align*}
LW(x)\leq M\cdot \mathbbm{1}_{O_N}(x), \qquad x\in E.
\end{align*}
Since $O_N$ is compact, we may apply \cite[Theorem 3.1]{mt} to conclude that the process $(X_t\colon t\geq 0)$ is non-evanescent, that is,
\begin{align}\label{noescape}
\mathbb{P}_x\left(\langle v,X_t\rangle\xrightarrow{t\rightarrow\infty}\infty\right)=0, \qquad x\in E.
\end{align}

Define the discrete time jump chain $(Y_n\colon n\in \mathbb{N}_0)$ by $Y_n=X_{J_n}$, where $J_0,J_1,\dots$ denote the jump times of $(X_t\colon t\geq 0)$ given by
\begin{align*}
J_0=0, \qquad J_{n+1}=\inf\{t\geq J_n: X_t\neq X_{J_n}\}.
\end{align*}
Let $B_m=\{Y_n\in O_m \text{ i.o.}\}$ and $F=\{Y_n\in A \text{ i.o.}\}=\{Y_n\in A \text{ for some $n$}\}$, where the last equality follows from $A$ being an absorbing set. By Assumption \ref{class}, all states in $O_m\subseteq E$ have a shortest path to $A$ (there is a path to the minimal irreducible class, and then to $A$, for any $x\in O_m$), which has some positive probability. For each state $x\in O_m$, let $b_x$ be the probability of this shortest path, and define $\beta_m=\min_{x\in O_m} b_x$. As $O_m$ is compact, $\beta_m>0$. It follows that for each $n\in \mathbb{N}_0$ the conditioned process fulfils
\begin{align*}
\mathbb{P}\left(\bigcup_{k=1}^\infty (Y_{n+k}\in A)\Big| Y_n\in O_m\right)\geq \beta_m>0.
\end{align*}
By \cite[Theorem 2.3]{durrett} we get
\begin{align}\label{trick}
\mathbb{P}_y\left(B_m\backslash F\right)=0,
\end{align}
with $y=Y_0=X_0=x$, for any $m\in \mathbb{N}$. Now, the complement of the event $\bigcup_{m=0}^\infty B_m$ is the event $G\cup F$, where $G=\left\{\langle v,Y_n\rangle \xrightarrow{n\rightarrow\infty} \infty\right\}$. As $B_m$ is an increasing sequence of events in $m$, we obtain by monotone convergence  and (\ref{trick}) that
\begin{align*}
1=\mathbb{P}_y\left(G\cup F\cup \bigcup_m B_m\right)=\lim_{m\to \infty}\mathbb{P}_y(G\cup F\cup B_m) = \lim_{m\to \infty}\mathbb{P}_y(G\cup F)=\mathbb{P}_y(G\cup F).
\end{align*}
Thus, $(Y_n\colon n\in \mathbb{N}_0)$ either tends to infinity or is eventually absorbed in $A$. The same holds for the full process $(X_t\colon t\geq 0)$, and by (\ref{noescape}) we conclude that $\mathbb P_x(G)=0$ hence $\mathbb{P}_x(\tau_A<\infty)=\mathbb{P}_x(X_t\in A \text{ for some $t$})=1$. In particular, we also have that
\begin{align*}
\lim_{n\to \infty}T_n=\tau_A,
\end{align*}
thus the process is regularly absorbed, by definition.
\end{proof}

Note that $\zeta$ in Proposition \ref{lem1} may be negative thus Assumption \ref{ass1}(i) is stronger and immediately provides the same conclusion of almost sure absorption of the process. Further, as we shall see in the next proposition, Assumption \ref{ass1}(ii) assures that the expected magnitude of $X_t$, in the form of $\langle v,X_t\rangle$ given $X_0=x$, is uniformly bounded in $x\in E$ for any $t>0$. This, in turn, implies that the time of ``coming down from infinity" is finite, which is closely related to the uniqueness of QSDs. This is where $\eta>0$ is required.

\begin{prop}\label{black}
Under Assumptions \ref{class}-\ref{ass1} with $A\neq \emptyset$, the process $(X_t,t\geq0)$ satisfies
\begin{align*}
\tau_A =\lim_{n\to \infty}T_n<\infty \qquad \mathbb{P}_x\text{-a.s. } \text{ for all $x\in E$},
\end{align*}
in particular, the process is absorbed $\mathbb{P}_x$-a.s. Further, $\sup_{x\in E}\mathbb{E}_x\langle v,X_t\rangle<\infty$ for any $t>0$.
\end{prop}
\begin{proof} By Assumption \ref{ass1}(i), it follows, applying the same notation as in Proposition \ref{lem1}, that
\begin{align*}
d_v(n)\geq  \eta d^v(n)>\zeta\frac{d^v(n)}{n},
\end{align*}
for $n$ sufficiently large and $\zeta$ as in the proposition. Thus by Proposition \ref{lem1}, the process $(X_t\colon t\geq 0)$ satisfies
\begin{align*}
\tau_A =\lim_{n\to \infty}T_n<\infty,
\end{align*}
$\mathbb{P}_x$-a.s. for all $x\in E$. Hence the process is regularly absorbed.

The second claim is apparently a `classical result' \cite{Expo} but we are not aware of a proof in the literature, hence we provide one here. Let $W(x)=\langle v,x\rangle$ on $\mathbb{N}_0^d$ as in Definition \ref{normlike}.  
It follows from (\ref{aqua}) of Proposition \ref{lem1} and Assumption \ref{ass1}(i) that with $\langle v,x\rangle=n$,
\begin{align*}
LW(x)&\leq  -d_v(n)+\zeta\max_{x'\in E, \langle v,x'\rangle=n}\sum_{k=1}^r \lambda_k(x')\mathbbm{1}_{A}(x'+\xi_k)\\
&\leq -d_v(n)+\zeta\frac{1}{n}d^v(n)\leq -d_v(n)+\frac{\zeta}{n \eta}d_v(n)=-\left(1-\frac{\zeta}{n\eta}\right)d_v(n).
\end{align*}
It follows, under Assumption \ref{ass1}(ii), that 
\begin{align*}
\frac{LW(x)}{W(x)^{1+\eta}}\leq -\left(1-\frac{\zeta}{n\eta}\right)\frac{d_v(n)}{n^{1+\eta}}\to -\infty,
\end{align*}
as $n=\langle v,x\rangle\to \infty$ in $\mathscr{N}$, in which case $1-\zeta/(n\eta)$ becomes positive. Hence, there exist constants $D_1,D_2>0$ such that
\begin{align}\label{eq}
LW(x)\leq D_2-D_1W(x)^{1+\eta}, \qquad \text{for all } x\in E.
\end{align}
Since we have $\sum_{k\in \mathcal{R}}\lambda_k(x)<\infty$ for each $x\in D$, it follows from \cite[p. 12]{stochreac} and the equivalence of the weakened and infinitesimal generators on $W$ that
\begin{align*}
W(X_{t})-W(0)-\int_0^{t}L W(X_s)\,ds
\end{align*}
is a martingale. Thus, by the martingale property we find
\begin{align}\label{dynkin}
\mathbb{E}_x W(X_{t})&=W(x)+\int_0^{t}\mathbb{E}_x(LW(X_s))\,ds,
\end{align}
which is a form of Dynkin's formula \cite{kallenberg}. Using the bound (\ref{eq}) combined with Jensen's inequality we obtain, upon differentiation of (\ref{dynkin}),
\begin{align*}
\frac{d}{dt}\mathbb{E}_xW(X_t)=\mathbb{E}_x(LW(X_t))\leq \mathbb{E}_x(D_2-D_1W(X_t)^{1+\eta})\leq D_2-D_1(\mathbb{E}_xW(X_t))^{1+\eta}.
\end{align*}
Define $f_x(t)=\mathbb{E}_xW(X_t)$ and choose $D_3>D_2$. Consider the associated differential equations
\begin{align}
g'_{x,\epsilon}(t)&=D_3-D_1 g_{x,\epsilon}(t)^{1+\eta}, \qquad g_{x,\epsilon}(0)=W(x)+\epsilon\label{ones}\\
h'_x(t)&=-D_1h_x(t)^{1+\eta}, \qquad h_x(0)=W(x)\label{twos}.
\end{align}
Define $F(t,z)=D_3-D_1z^{1+\eta}$. It then follows that
\begin{align*}
f'_x(t)&<F(t,f_x(t)), \qquad\quad f_x(0)=W(x),\\
h'_x(t)&<F(t,h_x(t)), \qquad\quad h_x(0)=W(x),\\
g'_{x,\epsilon}(t)&=F(t,g_{x,\epsilon}(t)), \qquad g_{x,\epsilon}(0)=W(x)+\epsilon.
\end{align*}
By Petrovitsch' theorem \cite[p. 316]{ineq} applied twice and the fact that solutions to the ordinary differential equations are continuous in the initial value, we conclude that
\begin{align}\label{ulige}
f_x(t)\leq g_{x}(t),  \qquad h_x(t)\leq g_{x}(t)\qquad t\in [0,T],
\end{align}
where $g_{x}(t)=\lim_{\epsilon\to 0}g_{x,\epsilon}(t)$. The solution to the initial value problem (\ref{ones})  cannot be given in explicit form, however, the associated simpler differential equation (\ref{twos}) does have an explicit solution for $\eta>0$, given by
\begin{align*}
h_x(t)=\frac{1}{(D_1\eta t+W(x)^{-\eta})^{1/\eta}}, \qquad t \geq 0.
\end{align*}
In order to use this to bound $f_x(t)$, define the function
\begin{align*}
k_t(x)=D_3t+h_x(t).
\end{align*}
Then, using (\ref{ulige}), we have
\begin{align*}
g'_{x}(t)=D_3-D_1g_{x}(t)^{1+\eta}\leq D_3-D_1h_x(t)^{1+\eta}=D_3+h_x'(t)=k_x'(t).
\end{align*}
Since $g_x(0)=W(x)=k_x(0)$ it follows that $g_x(t)\leq k_x(t)$ and we infer that
\begin{align*}
\sup_{x\in E}\mathbb{E}_x W(X_t)=\sup_{x\in E}f_x(t)\leq \sup_{x\in E}g_x(t)\leq \sup_{x\in E}k_x(t)<\infty,
\end{align*}
for all fixed $t>0$ as desired.
\end{proof}
\begin{defi}\label{defi1}
Let $v\in \mathbb{N}^d$ and $\alpha,\beta>1$. Define $V\colon\mathbb{N}_0^d\to \mathbb{R}$ and $\varphi\colon \mathbb{N}_0^d\to \mathbb{R}$ by
\begin{align*}
V(x)= \mathbbm{1}_{E}(x)\sum_{j=1}^{\langle v,x\rangle}\frac{1}{j^\alpha}, \qquad \varphi(x)= \mathbbm{1}_{E}(x)\sum_{j=\langle v,x\rangle+1}^{\infty}\frac{1}{j^\beta}.
\end{align*}
\end{defi}

\begin{lem} \label{lem2} Under Assumption \ref{ass1}, for suitable choices of $\alpha,\beta>1$, the pair $(V,\varphi)$ satisfies
\begin{enumerate}
 \item[(a)] $V, \varphi$ are bounded.
  \item[(b)] There exists an integer $n$ and a constant $C\geq0$, such that
  \begin{align*}
-L\varphi\leq C \mathbbm{1}_{O_n}.
\end{align*}
  \item[(c)] There exists constants $\epsilon, C'>0$ and $C''\geq0$, such that
  \begin{align*}
LV+C'\frac{V^{1+\epsilon}}{\varphi^{\epsilon}}\leq C''\varphi.
\end{align*}
\end{enumerate}
\end{lem}
\begin{proof} As $v\in \mathbb{N}^d$, it follows that $\langle v,x\rangle\in \mathbb{N}_0$ for all $x\in E$. In particular, $V(x),\varphi(x)\leq \sum_{j=1}^\infty \frac{1}{j}^p$ with $p>1$ which is a convergent hyperharmonic series. This proves (a). Thinking in terms of Riemann sums, using that $j\mapsto j^{-\beta}$ is decreasing in $j$ for $\beta>1$, we obtain the bound
\begin{align*}
\sum_{j=\langle v,x\rangle+1}^\infty\frac{1}{j^\beta}\leq \int_{\langle v,x\rangle}^\infty \frac{1}{y^\beta}\,dy=\frac{\langle v,x\rangle^{1-\beta}}{\beta-1}.
\end{align*}
Exploiting the linearity of $\langle v,\cdot \rangle$, we consider, for $x\in E$ with $n=\langle v,x\rangle$, and $L$ the weakened generator which coincides with the infinitesimal generator since $\varphi$ is bounded,
\begin{align*}
L\varphi(x)=&\sum_{k=1}^r\lambda_k(x)\left(\varphi(x+\xi_k)-\varphi(x)\right)\\
=&\sum_{k=1}^r\lambda_k(x)\mathbbm{1}_{E}(x+\xi_k)\left(\varphi(x+\xi_k)-\varphi(x)\right)+\sum_{k=1}^r\lambda_k(x)\mathbbm{1}_{A}(x+\xi_k)\left(\varphi(x+\xi_k)-\varphi(x)\right)\\
=&\sum_{k=1}^r\lambda_k(x)\mathbbm{1}_{E}(x+\xi_k)\left(\sum_{j=\langle v, x+\xi_k\rangle+1}^\infty \frac{1}{j^\beta}-\sum_{j=\langle v, x\rangle+1}^\infty \frac{1}{j^\beta}\right)\\
&-\sum_{k=1}^r\lambda_k(x)\mathbbm{1}_{A}(x+\xi_k)\sum_{j=\langle v,x\rangle+1}^\infty\frac{1}{j^\beta}\\
=&-\sum_{k=1}^r\lambda_k(x)\mathbbm{1}_{E}(x+\xi_k)\mathbbm{1}_{(0,\infty)}(\langle v,\xi_k\rangle)\sum_{i=1}^{\langle v,\xi_k\rangle}\frac{1}{(\langle v,x\rangle+i)^\beta}\\
&+\sum_{k=1}^r\lambda_k(x)\mathbbm{1}_{E}(x+\xi_k)\mathbbm{1}_{(-\infty,0)}(\langle v,\xi_k\rangle)\sum_{i=\langle v,\xi_k\rangle}^{-1}\frac{1}{(\langle v,x\rangle+i+1)^\beta}\\
&-\sum_{k=1}^r\lambda_k(x)\mathbbm{1}_{A}(x+\xi_k)\sum_{j=\langle v,x\rangle+1}^\infty\frac{1}{j^\beta}\\
\geq &  \frac{1}{\langle v,x\rangle^\beta}\left(-\sum_{k=1}^r\lambda_k(x)\mathbbm{1}_{E}(x+\xi_k)\mathbbm{1}_{(0,\infty)}(\langle v,\xi_k\rangle)\langle v,\xi_k\rangle\right.\\
&\left.-\sum_{k=1}^r\lambda_k(x)\mathbbm{1}_{E}(x+\xi_k)\mathbbm{1}_{(-\infty,0)}(\langle v,\xi_k\rangle)\langle v,\xi_k\rangle\right)-\sum_{k=1}^r\lambda_k(x)\mathbbm{1}_{A}(x+\xi_k)\frac{\langle v,x\rangle^{1-\beta}}{\beta-1}\\
=& \frac{1}{\langle v,x\rangle^\beta}\left(-\sum_{k=1}^r\lambda_k(x)\mathbbm{1}_{E}(x+\xi_k)\langle v,\xi_k\rangle-\sum_{k=1}^r\lambda_k(x)\mathbbm{1}_{A}(x+\xi_k)\frac{\langle v,x\rangle}{\beta-1}\right)\\
\geq&\frac{1}{\langle v,x\rangle^\beta}\left(-\max_{x'\in E, \langle v,x'\rangle=n}\sum_{k=1}^r\lambda_k(x')\langle v,\xi_k\rangle\mathbbm{1}_{E}(x'+\xi_k)\right.\\
&\left.-\max_{x'\in E, \langle v,x'\rangle=n}\sum_{k=1}^r\lambda_k(x')\mathbbm{1}_{A}(x'+\xi_k)\frac{\langle v,x'\rangle}{\beta-1}\right)\\
=&\frac{1}{n^\beta}\left(d_v(n)-\frac{d^v(n)}{\beta-1}\right).
\end{align*}
Using Assumption \ref{ass1}(i), we can choose $\beta>1$ large enough such that $L\varphi(x)\geq0$ for all $n=\langle v,x\rangle$ sufficiently large. In particular, condition (b) is satisfied.
\\\\
Similarly, as $V$ is bounded, we find for $x\in E$ with $\langle x,v\rangle=n$,
\begin{align*}
LV(x)=&\sum_{k=1}^{r}\lambda_k(x)\left(V(x+\xi_k)-V(x)\right)\\
=&\sum_{k=1}^r\lambda_k(x)\mathbbm{1}_{E}(x+\xi_k)\left(\sum_{j=1}^{\langle v,x+\xi_k\rangle}\frac{1}{j^\alpha}-\sum_{j=1}^{\langle v,x\rangle}\frac{1}{j^\alpha}\right)-\sum_{k=1}^r\lambda_k(x)\mathbbm{1}_{A}(x+\xi_k)\sum_{j=1}^{\langle v,x\rangle }\frac{1}{j^\alpha}\\
\leq& \sum_{k=1}^r \lambda_k(x)\mathbbm{1}_{E}(x+\xi_k)\mathbbm{1}_{(0,\infty)}(\langle v,\xi_k\rangle)\sum_{j=\langle v,x\rangle+1}^{\langle v,x+\xi_k\rangle}\frac{1}{j^\alpha}\\
&-\sum_{k=1}^r\lambda_k(x)\mathbbm{1}_{E}(x+\xi_k)\mathbbm{1}_{(-\infty,0)}(\langle v,\xi_k\rangle)\sum_{j=\langle v,x+\xi_k\rangle+1}^{\langle v,x\rangle}\frac{1}{j^\alpha}\\
=&\sum_{k=1}^r \lambda_k(x)\mathbbm{1}_{E}(x+\xi_k)\mathbbm{1}_{(0,\infty)}(\langle v,\xi_k\rangle)\sum_{j=1}^{\langle v,\xi_k\rangle}\frac{1}{(j+\langle v,x\rangle)^\alpha}\\
&-\sum_{k=1}^r\lambda_k(x)\mathbbm{1}_{E}(x+\xi_k)\mathbbm{1}_{(-\infty,0)}(\langle v,\xi_k\rangle)\sum_{j=\langle v,\xi_k\rangle}^{-1}\frac{1}{(j+\langle v,x\rangle+1)^\alpha}\\
\leq& \sum_{k=1}^r \lambda_k(x)\mathbbm{1}_{E}(x+\xi_k)\mathbbm{1}_{(0,\infty)}(\langle v,\xi_k\rangle)\frac{\langle v,\xi_k\rangle}{\langle v,x\rangle^\alpha}\\
&+\sum_{k=1}^{r}\lambda_k(x)\mathbbm{1}_{E}(x+\xi_k)\mathbbm{1}_{(-\infty,0)}(\langle v,\xi_k\rangle)\frac{\langle v,\xi_k\rangle}{\langle v,x\rangle^\alpha}\\
\leq& \frac{1}{\langle v,x\rangle^\alpha}\sum_{k=1}^r\lambda_k(x)\langle v,\xi_k\rangle\mathbbm{1}_{E}(x+\xi_k)
\\
\leq&\frac{1}{\langle v,x\rangle^\alpha}\max_{x'\in E, \langle v,x'\rangle=n}\sum_{k=1}^r\lambda_k(x')\langle v,\xi_k\rangle\mathbbm{1}_{E}(x'+\xi_k)= -\frac{d_v(n)}{n^\alpha}.
\end{align*}
Note that by treating $V(x)$ as a lower Riemann sum, for $x\in E$,
\begin{align*}
V(x)&\leq \sum_{j=1}^{\infty}\frac{1}{j^\alpha}=1+\sum_{j=2}^{\infty}\frac{1}{j^\alpha}\leq 1+\int_1^\infty \frac{1}{x^\alpha}\,dx=\frac{\alpha}{\alpha-1},
\end{align*}
and similarly, treating $\varphi(x)$ as an upper Riemann sum,
\begin{align*}
\varphi(x)&=\sum_{j=\langle v,x\rangle+1}^\infty\frac{1}{j^\beta}\geq\int_{\langle v,x\rangle+1}^\infty \frac{1}{x^\beta}\,dx=\frac{(1+\langle v,x\rangle)^{1-\beta}}{\beta-1}\geq \frac{\langle v,x\rangle^{1-\beta}}{2(\beta-1)},
\end{align*}
with the last inequality holding for $\langle v,x\rangle$ sufficiently large, using that for $\beta>1$,
\begin{align*}
\frac{(1+\langle v,x\rangle)^{1-\beta}}{\langle v,x\rangle^{1-\beta}}=\left(\frac{1}{\langle v,x\rangle}+1\right)^{1-\beta}\to 1>\frac{1}{2}, \qquad \text{for }\quad n=\langle v,x\rangle \to \infty.
\end{align*}
We infer that
\begin{align*}
LV(x)+\frac{V^{1+\epsilon}(x)}{\varphi^\epsilon(x)}\leq -\frac{d_v(n)}{n^\alpha}+Cn^{\epsilon(\beta-1)},
\end{align*}
where $C=[\alpha/(\alpha-1)]^{1+\epsilon}[2(\beta-1)]^\epsilon$. Note that by definition $\varphi(x)>0$ for $x\in E$ hence, choosing $\alpha=1+\eta/2$ and $\epsilon=\eta/[2(\beta-1)]$, we get
\begin{align*}
LV(x)+\frac{V^{1+\epsilon}(x)}{\varphi^\epsilon(x)}\leq -\frac{d_v(n)}{n^{1+\eta/2}}+Cn^{\eta/2}=\left(C-\frac{d_v(n)}{n^{1+\eta}}\right)n^{\eta/2}.
\end{align*}
Using Assumption \ref{ass1}(ii), the first term becomes negative for $\langle v,x\rangle$ sufficiently large. In other words $LV(x)+\frac{V^{1+\epsilon}(x)}{\varphi^\epsilon(x)}\leq 0$ for $x\notin O_n$ with $n$ sufficiently large. Since by definition $\varphi(x)\geq0$, condition (c) holds.
\end{proof}
\begin{lem}\label{lem3} Under Assumpstions \ref{class}-\ref{ass1}, 
the pair $(V,\varphi)$ is an admissible couple of functions.
\end{lem}
\begin{proof}
Choose $\alpha,\beta\in \mathbb{R}$ such that the conclusions of Lemma \ref{lem2} hold. In particular, $\alpha,\beta>1$ hence the functions $V$ and $\varphi$ are bounded, non-negative on $E\cup A$ and positive on $E$, and by definition $V(x)=\varphi(x)=0$ for $x\in A$. Furthermore, $\inf_{x\in E} V(x)>0$, hence by non-negativity of $\varphi$,
\begin{align*}
\inf_{x\in E}\frac{V(x)}{\varphi(x)}>0,
\end{align*}
and Definition \ref{couple}(1) is fulfilled.
Let $(x_p)_{p\geq1}$ be any sequence in $E$ such that the set $\{p\in \mathbb{N}: x_p\in O_n\}$ is finite for all $n\geq1$. Then $\langle v,x_p\rangle\to \infty$ as $p\to \infty$, hence
\begin{align*}
\lim_{p\to \infty}\frac{V(x_p)}{\varphi(x_p)}=\lim_{p\to \infty}\frac{\sum_{j=1}^{\langle v,x_p\rangle}\frac{1}{j^\alpha}}{\sum_{j=\langle v,x_p\rangle+1}^{\infty}\frac{1}{j^\beta}}=\infty.
\end{align*}
Furthermore, since the process is regularly absorbed by Proposition \ref{black}, we have
\begin{align*}
\lim_{n\to \infty}V(X_{T_n})=V(X_{\tau_A})=0 \qquad \mathbb{P}_x\text{-}a.s,
\end{align*}
hence Definition \ref{couple}(2) is fulfilled. Finally, from Lemma \ref{lem2} and the fact that $V$ and $\varphi$ are both bounded functions, it follows that $LV$ is bounded from above and $L\varphi$ is bounded from below. This concludes the proof.
\end{proof}

\subsection{Lemmas}
\begin{lem}\label{first}
Assumption \ref{class}(i) is equivalent to the following: There exists $n_0\in \mathbb{N}$, $\theta_0,\theta_1,a_1>0$ and a probability measure $\nu$ on $E$ such that, for all $x\in O_{n_0}$ and all $s\in [\theta_0,\theta_0+\theta_1]$,
\begin{align}\label{firstone}
\mathbb{P}_x(X_s\in \cdot) \geq a_1\nu,
\end{align}
and in addition, for all $n\geq n_0$, there exists $s_n\geq 0$ such that
\begin{align}\label{secondone}
\inf_{x\in O_n} \mathbb{P}_x(X_{s_n}\in O_{n_0})>0.
\end{align}

\end{lem}
\begin{proof}
We first prove that Assumption \ref{class} implies the existence of such constants and probability measure. Let $\mathcal I_{\min}$ be the unique minimal irreducible class contained in $E$. Set
\begin{align*}
n_0=\inf\{n\in \mathbb{N}\colon O_n\cap \mathcal I_{\min}\neq \emptyset\}<\infty.
\end{align*}
Pick $z\in \mathcal I_{\min}\cap O_{n_0}$ arbitrarily and let $\nu=\delta_z$. Pick arbitrary $\theta_0,\theta_1>0$ and let
\begin{align*}
a_1=\inf_{s\in [\theta_0,\theta_0+\theta_1], x\in O_{n_0}}\mathbb{P}_x(X_s=z).
\end{align*}
Note that for all $n\in \mathbb{N}$ the set $O_{n}$ is finite and any $x\in O_{n}\subset E$ has a path to $\mathcal I_{\min}$ and thus, by irreducibility, to $z$. By continuity of $\mathbb{P}_x(X_\cdot=z)$, we conclude that $a_1>0$ and choosing $s_n=\theta_0>0$ for all $n$, the desired holds.

\medskip
For the reverse direction, we first prove uniqueness of the minimal irreducible class in the endorsed set $E$. Suppose for contradiction that $\mathcal{I}_i\neq \mathcal{I}_j$ are irreducible minimal classes in $E$. Let $x_1\in \mathcal{I}_i$ and $x_2\in \mathcal{I}_j$. Then there is a path $x_1\mapsto y_1$ and a path $x_2\mapsto y_2$ with $y_1,y_2\in O_{n_0}$. Indeed, if $\mathcal{I}_\ell\cap O_{n_0}\neq\emptyset$ for $l=i,j$ we may simply choose $y_\ell=x_\ell$, $\ell=1,2$. Otherwise, $x_1,x_2$ are in some $O_{n_1}$ and $O_{n_2}$ respectively with $n_1,n_2>n_0$, hence by (\ref{secondone}), there exist paths as described. By (\ref{firstone}), there exist $\theta_0,\theta_1,a_1>0$ and a probability measure $\nu$ on $E$ such that, for all $y\in O_{n_0}$ and all $s\in [\theta_1,\theta_1+\theta_0]$,
\begin{align*}
\mathbb{P}_y(X_s\in \cdot) \geq a_1\nu.
\end{align*}
As $\nu$ is a probability measure on a countable space, there exists some $z\in E$ such that $\nu(\{z\})>0$. But then $\mathbb{P}_y(X_s=z)>0$ for all $y\in O_{n_0}$. We conclude that there exist paths $y_1\mapsto z$ and $y_2\mapsto z$. By minimality of $\mathcal{I}_i$ and $\mathcal{I}_j$, we conclude that $z\in \mathcal{I}_i\cap \mathcal{I}_j$ which is a contradiction. This proves the uniqueness in Assumption \ref{class}(i).

\medskip
We now prove existence. Suppose for contradiction that no minimal class exists. Then, for all $\mathcal{I}\subseteq E$ there exists $\mathcal{J}\subseteq E$ such that $\mathcal{I}\to \mathcal{J}$. Repeating the argument results in an infinite path
\begin{align}\label{infinitepath}
\mathcal{J}_1\to \mathcal{J}_2 \to  \mathcal{J}_3\to \dots
\end{align}
In the case where there are only finitely many irreducible classes in $E$, there must exist an $i\in \mathbb{N}$ such that $\mathcal{J}_1\to \mathcal{J}_i=\mathcal{J}_1$ which contradicts the lack of cycles in $\mathcal{D}$.

\medskip
Suppose therefore that there are infinitely many classes. Given $n_0$, the set $O_{n_0}$ is finite. Thus it intersects at most finitely many irreducible classes. We conclude that in the infinite path (\ref{infinitepath}), there are infinitely many $i\in \mathbb{N}$ such that
\begin{align*}
\mathcal{I}_i\cap O_{n_0}=\emptyset, \qquad \mathcal{I}_i\cap O_{n_i}\neq \emptyset,
\end{align*}
for some $n_i> n_0$. However, by (\ref{secondone}) each of these classes have a path to $O_{n_0}$. This implies the existence of at least one irreducible class intersecting $O_{n_0}$ which appears more than once in the infinite path (\ref{infinitepath}). This creates a cycle in $\mathcal{D}$ which is a contradiction.
\end{proof}

\begin{lem}
Under Assumpstions \ref{class}-\ref{ass1}, for all $\lambda>0$, there exists $n\geq 1$ such that
\begin{align}\label{eq4}
\sup_{x\in E}\mathbb{E}_x\left(e^{\lambda(\tau_{n}\wedge\tau_A)}\right)<\infty.
\end{align}
\end{lem}
\begin{proof}
It follows from Proposition \ref{black} that $\sup_{x\in E}\mathbb{E}_x\langle v,X_t\rangle< M$ for all $t>0$ and some constant $M>0$. Let $\tau_{n,A}=\tau_{n}\wedge \tau_{A}$. Then, for $\epsilon>0$
\begin{align*}
\mathbb{E}_x\langle v,X_\epsilon\rangle=\sum_{n=0}^\infty \mathbb{P}_x(\langle v,X_\epsilon\rangle>n)\geq \sum_{n=0}^\infty \mathbb{P}_x(\tau_{n,A}>\epsilon).
\end{align*}
\noindent
Suppose for contradiction that $\mathbb{P}_x(\tau_{n,A}>\epsilon)$ does not converge uniformly in $x$ to 0 as $n\to \infty$. Then for any $\delta>0$, $\sup_{x\in E}\mathbb{P}_x(\tau_{n,A}>\epsilon)>\delta$ for infinitely many $n\in\mathbb{N}$. Thus, choosing $\delta=\epsilon$, there exists a sequence $(x_i,n_i)_{i\geq 1}$ with $\langle v,x_i \rangle>n_i$ such that $\mathbb{P}_{x_i}(\tau_{n_i,A}>\epsilon)>\epsilon$ and $n_i>n_{i-1}$. But then, noting that  $\mathbb{P}_x(\tau_{n-1,A}>\epsilon)>\mathbb{P}_x(\tau_{n,A}>\epsilon)$ for all $n\in \mathbb{N}$, we obtain
\begin{align*}
\mathbb{E}_{x_i}\langle v,X_\epsilon\rangle\geq \sum_{n=0}^\infty \mathbb{P}_{x_i}(\tau_{n,A}>\epsilon)\geq \sum_{n=0}^{n_i} \mathbb{P}_{x_i}(\tau_{n,A}>\epsilon)\geq n_i \mathbb{P}_{x_i}(\tau_{n_i,A}>\epsilon)>n_i\epsilon.
\end{align*}
Letting $i\to \infty$ we would have $M>\lim_{i\to \infty}\mathbb{E}_{x_i}\langle v,X_\epsilon\rangle>\lim_{i\to \infty} n_i\epsilon=\infty$,
which is a contradiction. We conclude that $\sup_{x\in E}\mathbb{P}_x(\tau_{n,A}>\epsilon)\leq\epsilon$ for $n\geq N(\epsilon)$.

\medskip
Now, given $t>0$ we can choose $0<\epsilon< 1$ such that $t=p\epsilon$ for some $p\in \mathbb{N}$. Applying the Markov property we get, for $n\geq N(\epsilon)$,
\begin{align*}
&\mathbb{P}_x(\tau_{n,A}\geq t)=\mathbb{P}_x(\tau_{n,A}\geq p\epsilon)\\
&=\sum_{y\in D}\mathbb{P}_x(\tau_{n,A}>p\epsilon\,|\, \tau_{n,A}>(p-1)\epsilon, X_{(p-1)\epsilon}=y)\mathbb{P}_x(\tau_{n,A}>(p-1)\epsilon, X_{(p-1)\epsilon} =y)\\
&=\sum_{y\in D}\mathbb{P}_y(\tau_{n,A}>\epsilon)\mathbb{P}_x(\tau_{n,A}>(p-1)\epsilon, X_{(p-1)\epsilon} =y)\leq \epsilon \,\mathbb{P}_x(\tau_{n,A}>(p-1)\epsilon)\leq \epsilon^p,
\end{align*}
hence we conclude, as $t< p$, that
\begin{align*}
\sup_{x\in E}\mathbb{P}_x(\tau_{n,A}\geq t)\leq \epsilon^t.
\end{align*}
Since $e^{\lambda (\tau_{n}\wedge \tau_A)}$ is non-negative, by choosing $\epsilon<e^{-\lambda}<1$ we get
\begin{align*}
\sup_{x\in E}\mathbb{E}_x&\left(e^{\lambda (\tau_{n}\wedge \tau_A)}\right)=\sup_{x\in E}\int_{0}^\infty \mathbb{P}_x\left(e^{\lambda (\tau_{n}\wedge \tau_{A})}\geq t\right)\,dt\leq 1+\sup_{x\in E}\int_{1}^\infty \mathbb{P}_x\left(\tau_{n,A}\geq \frac{\ln t}{\lambda}\right)\,dt\\
&\quad\leq 1+\int_1^\infty \epsilon^{\ln t/\lambda}\,dt=1+\int_0^\infty (e\epsilon^{1/\lambda})^u\,du=1+\frac{1}{1+\ln(\epsilon)/\lambda}<\infty,
\end{align*}
thus (\ref{eq4}) holds as desired.
\end{proof}
\begin{lem}\label{nr3}
For all $n\geq 0$, there exists a constant $C_n$ such that, for all $t\geq 0$,
\begin{align*}
\sup_{x\in O_n}\mathbb{P}_x(t<\tau_A) \leq C_n \inf_{x\in O_n}\mathbb{P}_x(t<\tau_A).
\end{align*}
Further, with $V$ from Definition \ref{defi1}, there exist constants $r_0,p_0>0$ such that for $n$ sufficiently large,
\begin{align*}
\mathbb{P}_x(r_0<\tau_A)\leq p_0 V(x), \qquad \text{for all } x\in E\backslash O_n.
\end{align*}
\end{lem}
\begin{proof}
Let $n\geq0$ be given. If $O_n$ is empty, then the statement is vacuously true for any $C_n$. If $O_n$ is non-empty, then as $\sup_{x\in O_n}\mathbb{P}_x(t<\tau_A)\leq 1$ and $\inf_{x\in O_n}\mathbb{P}_x(t<\tau_A)>0$ since $O_n$ is finite, we may simply choose
\begin{align*}
C_n=\frac{\sup_{x\in O_n}\mathbb{P}_x(t<\tau_A)}{\inf_{x\in O_n}\mathbb{P}_x(t<\tau_A)}<\infty.
\end{align*}
To see the second claim, note that letting
\begin{align*}
p_0=\frac{1}{\inf_{x\in E}V(x)}>0,
\end{align*}
we have for any $r_0>0$,
\begin{align*}
\mathbb{P}_x(r_0<\tau_A)\leq 1=p_0 \inf_{x\in E} V(x)\leq p_0 V(x),
\end{align*}
for all $x\in E$, and the desired holds.
\end{proof}

\section{The Main Result}
We are now ready to state and prove the main result of the paper. In the case where only a single endorsed set, $E$, is considered, we find that the unique QSD hereon is in fact globally attracting in the space of probability measures on $E$.

\begin{theo}\label{res}
A reaction network $(\mathcal{N},\lambda)$ with associated stochastic process $(X_t\colon t\geq 0)$ on $D=E\sqcup A$, with $A\neq \emptyset$ and satisfying Assumption \ref{class}-\ref{ass1}, admits a unique quasi-stationary distribution $\nu$. Further, there exist constants $C,\gamma>0$ such that, for all probability measures $\mu$ on $E$,
\begin{align*}
\|\mathbb{P}_\mu(X_t\in \cdot\,|\, t<\tau_A)-\nu\|_{TV}\leq Ce^{-\gamma t}, \qquad t\geq0.
\end{align*} 
\end{theo}
\begin{proof}
By Proposition \ref{black}, the process is regularly absorbed, and by Lemma \ref{lem3} the pair $(V,\varphi)$ given in Definition \ref{defi1} is admissible, satisfying conditions (a) and (b) from Lemma \ref{lem2}. The result now follows from Lemma \ref{first}-\ref{nr3} together with \cite[Cor. 2.8]{cd}.
\end{proof}

\begin{kor}\label{koret}
Let $\nu$ be the unique quasi-stationary distribution on $E$ and $\mathcal{I}_{\text{min}}$ the unique minimal class of $E$. Then $\text{supp }\nu=\mathcal{I}_{\text{min}}$.
\end{kor}
\begin{proof}
Suppose for contradiction that $\text{supp } \nu\not\subseteq \mathcal{I}_{\text{min}}$. Then there exists a point $y\in \mathcal{I}\neq \mathcal{I}_{\text{min}}$ for which $\nu(\{y\})>0$, where $\mathcal{I}$ is an irreducible class of $E$. As $\nu$ is globally attracting in $\mathcal{P}(E)$, the space of probability distributions on $E$, it follows by definition that 
\begin{align*}
\lim_{t\to \infty}\mathbb{P}_\mu(X_t\in B\,|\, t<\tau_A)=\nu(B),
\end{align*}
for any $\mu\in \mathcal{P}(E)$ and any measurable set $B\subseteq E$. In particular, letting $\mu=\delta_z$ with $z\in \mathcal{I}_{\text{min}}$ yields, by minimality
\begin{align*}
0=\lim_{t\to \infty}\mathbb{P}_{\delta_z}(X_t=y\,|\, t<\tau_A)=\nu(\{y\})>0,
\end{align*}
which is a contradiction. We conclude that $\text{supp } \nu\subseteq \mathcal{I}_{\text{min}}$. In particular, there exists $x'\in \mathcal{I}_{\text{min}}$ such that $\nu(\{x'\})>0$. Further, since $\mathcal{I}_{\text{min}}$ is irreducible, $\mathbb{P}_x(X_t=y)>0$ for all $x,y\in\mathcal{I}_{\text{min}}$. As $\nu$ is a QSD, it follows from \cite[p. 48]{QSD} that
\begin{align*}
e^{-\theta t}\nu(y)=\sum_{x\in E}\nu(x)\mathbb{P}_x(X_t=y)=\sum_{x\in\mathcal{I}_{\text{min}}}\nu(x)\mathbb{P}_x(X_t=y)>\nu(x')\mathbb{P}_{x'}(X_t=y)>0.
\end{align*}
for some $\theta>0$ and all $t\geq0$, $y\in \mathcal{I}_{\text{min}}$. Consequently, $\text{supp } \nu\supseteq \mathcal{I}_{\text{min}}$ which in turn implies $\text{supp } \nu=\mathcal{I}_{\text{min}}$ as desired.
\end{proof}

We now examine the general holistic setting with state space $D=D_E\cup D_A$, where $D_E$ consists of possibly several endorsed sets, each with or without a corresponding non-empty absorbing set. As one might in practice not have complete information about the starting-point of the process, one may in general not know exactly which endorsed set the process evolves in. However, one may have a qualitative guess in the form of an initial distribution, $\mu$ on $D_E$.

\medskip
The following theorem shows that we may consider the problem of finding a unique QSD on each endorsed set independently and then piecing these together to form a unique limiting measure up to a choice of the initial distribution, $\mu$, on $D_E$.

\begin{theo}\label{samlet}
Let $(\mathcal{N},\lambda)$ be a reaction network and $U=E_1\cup\dots\cup E_{m}$ a finite union of endorsed sets. If Assumption \ref{class}-\ref{ass1}  are satisfied, then the associated stochastic process $(X_t\colon t\geq 0)$ admits a unique quasi-stationary distribution, $\nu_n$, on each endorsed set $E_n\subseteq U$, $n=1,\ldots,m$. Furthermore, given an initial distribution $\mu$ on $U$, the measure $\nu_\mu$ defined by
\begin{align*}
\nu_\mu(B)=\sum_{n=1}^m\mu(E_n) \nu_n(B\cap E_n),
\end{align*}
is well defined and there exist constants $C,\gamma>0$ such that,
\begin{align*}
\|\mathbb{P}_\mu(X_t\in \cdot\,|\, t<\tau_A)-\nu_{\mu}\|_{TV}\leq Ce^{-\gamma t}, \qquad t\geq0.
\end{align*} 
\end{theo}
\begin{proof}
That $(X_t\colon t\geq 0)$ admits a unique quasi-stationary distribution, $\nu_n$, on each endorsed set $E_n\subseteq U$ with $A_n\neq \emptyset$ follows from Theorem \ref{res}. Now, suppose $A_n=\emptyset$ for some $n$. Then the definitions of quasi-stationary distribution and stationary distribution on $E_n$ are equal, $\tau_{A_n}=\infty$ and all the previous proofs go through without changes. In particular, it follows from Corollary \ref{koret} that any stationary distribution, $\pi_n$, is supported by the unique minimal irreducible class as well. Furthermore, by Theorem \ref{res}, for any $\mu\in \mathcal{P}(E_n)$ there exists $C_1,\gamma_1>0$ such that for any $B\subseteq E_n$
\begin{align*}
\|\mathbb{P}_{\mu}(X_t\in B)-\pi_n(B)\|_{TV}&=\|\mathbb{P}_{\mu}(X_t\in B\,|\, t<\tau_A)-\nu_n(B)\|_{TV}\leq C_1e^{-\gamma_1 t}.
\end{align*}
Clearly, $\nu_{\mu}$ is a well defined probability measure. Finally, we may let
\begin{align*}
C=\max_{n\in \{1,\dots,m\}}{C_n}>0, \qquad \gamma=\min_{n\in \{1,\dots,m\}}\gamma_n>0,
\end{align*}
where $C_n,\gamma_n>0$ for $n\in \{1,\dots,m\}$ are given through Theorem \ref{res} and the above argument, for each class $E_n$. This proves the desired.
\end{proof}

\begin{kor}
A one-species reaction network $(\mathcal{N},\lambda)$ has $gcd(\xi_1,\dots, \xi_r)<\infty$ endorsed sets. If in addition the kinetics $\lambda$ is stochastic mass-action, up to a choice of the initial distribution $\mu$ on $D_E$, there is a unique QSD on $D_E$ if the highest order of a reaction is at least 2 and
\begin{align*}
\sum_{k\in \mathcal{R}^\ast}\lambda_k(x)\xi_k<0, \qquad \text{for } \langle v,x\rangle \text{ sufficiently large},
\end{align*}
where $\mathcal{R}^\ast$ is the set of reactions of highest order.
\end{kor}
\begin{proof}
Note first that by Corollary \ref{dim2}, any one-species reaction network has a finite number of endorsed sets, as $\rank \,\Xi=1$. In fact, the number is given by $\text{gcd}(\xi_1,\dots,\xi_r)$. Indeed, as $\text{gcd}(\xi_1,\dots,\xi_r)\leq |\xi_k|$ it follows that $\mathcal{P}(\text{gcd}(\xi_1,\dots,\xi_r))\cap \mathcal{L}(\Xi)=\emptyset$ where $\mathcal{P}(B)$ is the fundamental parallelepiped generated by $B$ and $\mathcal{L}(\Xi)$ is the lattice generated by $\Xi$ as introduced in the proof of Proposition \ref{state}. Therefore, $b_1=\text{gcd}(\xi_1,\dots,\xi_r)$ is a basis for $\mathcal{L}(\Xi)$ and we conclude that the number of endorsed sets is $|\det(b_1)|=b_1$ \cite{lattice}.

As there is only one species, we may without loss of generality take $v=1$. It follows that $d_v(n)=-\sum_{k\in \mathcal{R}} \lambda_k(n)\xi_k>0$ for $n$ sufficiently large and $d_v(n)=\mathcal{O}(n^a)$ with $a\geq 2$ by assumption. Furthermore, $d^v(n)=0$ for $n$ sufficiently large. Thus Assumption \ref{ass1} is satisfied. We infer the desired by Theorem \ref{samlet}.
\end{proof}

\section{Examples}
In this section, the main theorems and their applicability are illustrated through a series of examples. In particular, we show explicitly how the results of \cite{cd} are extended.

\begin{eks}
\vspace{0.2cm}
\begin{multicols}{2}
\begin{align*}
mS_1\overset{\alpha_1}{\to} \emptyset
\end{align*}
\columnbreak

\phantom.
\vspace{0.1cm}
\begin{center}
\vspace{-0.2cm}
 \def\svgwidth{140pt}
\begingroup%
  \makeatletter%
  \providecommand\color[2][]{%
    \errmessage{(Inkscape) Color is used for the text in Inkscape, but the package 'color.sty' is not loaded}%
    \renewcommand\color[2][]{}%
  }%
  \providecommand\transparent[1]{%
    \errmessage{(Inkscape) Transparency is used (non-zero) for the text in Inkscape, but the package 'transparent.sty' is not loaded}%
    \renewcommand\transparent[1]{}%
  }%
  \providecommand\rotatebox[2]{#2}%
  \ifx\svgwidth\undefined%
    \setlength{\unitlength}{114.10339499bp}%
    \ifx\svgscale\undefined%
      \relax%
    \else%
      \setlength{\unitlength}{\unitlength * \real{\svgscale}}%
    \fi%
  \else%
    \setlength{\unitlength}{\svgwidth}%
  \fi%
  \global\let\svgwidth\undefined%
  \global\let\svgscale\undefined%
  \makeatother%
  \begin{picture}(1,0.12708088)%
    \put(0,0){\includegraphics[width=\unitlength,page=1]{eks11grid.pdf}}%
    \put(0.85212661,0.01033448){\color[rgb]{0,0,0}\makebox(0,0)[lb]{\smash{$S_1$}}}%
    \put(0,0){\includegraphics[width=\unitlength,page=2]{eks11grid.pdf}}%
  \end{picture}%
\endgroup%

\end{center}
\end{multicols}
\vspace{-0.2cm}
\noindent
As discussed in Section 2, the endorsed sets and corresponding absorbing sets are
\begin{align*}
E_i=\{i+pm-1 \,|\, p\in \mathbb{N}\}, \qquad A_i=\{i-1\}, \qquad \text{for } i=1,\dots,m,
\end{align*}
respectively. These endorsed sets are evidently not irreducible. However, $\{i+m-1\}$ is the unique minimal irreducible class in $E_i$ from which one may jump directly to $A_i$, thus Assumption \ref{class} is satisfied. Further, assuming mass-action kinetics,
\begin{align*}
d_v(n)&=\alpha_1 n(n-1)\dots (n-m+1)m\mathbbm{1}_E(n-m)=\mathcal{O}(n^m),\\
d^v(n)&=\mathcal{O}(1),
\end{align*}
for $n$ sufficiently large, hence we conclude by Theorem \ref{samlet} that there exists a unique QSD on each $E_i$ if $m\geq 2$. Further, in this case, for any initial distribution, $\mu$, on the full set of endorsed states, $D_E=\{m,m+1,\dots\}$, the measure $\mathbb{P}_\mu(X_t\in \cdot\,|\, t<\tau_A)$ tends to $\nu_\mu$ exponentially fast for $t\to \infty$.
\end{eks}

\begin{eks}[Lotka-Volterra]
\phantom{.}
\begin{multicols}{2}
\begin{align*}
S_1&\overset{\alpha_1}{\rightarrow} 2S_1\\
S_1+S_2& \overset{\alpha_2}{\rightarrow} 2S_2\\
S_2&\overset{\alpha_3}{\to} \emptyset
\end{align*}
\columnbreak
\vspace{0.4cm}
\begin{center}
 \def\svgwidth{130pt}
\begingroup%
  \makeatletter%
  \providecommand\color[2][]{%
    \errmessage{(Inkscape) Color is used for the text in Inkscape, but the package 'color.sty' is not loaded}%
    \renewcommand\color[2][]{}%
  }%
  \providecommand\transparent[1]{%
    \errmessage{(Inkscape) Transparency is used (non-zero) for the text in Inkscape, but the package 'transparent.sty' is not loaded}%
    \renewcommand\transparent[1]{}%
  }%
  \providecommand\rotatebox[2]{#2}%
  \ifx\svgwidth\undefined%
    \setlength{\unitlength}{116.48636904bp}%
    \ifx\svgscale\undefined%
      \relax%
    \else%
      \setlength{\unitlength}{\unitlength * \real{\svgscale}}%
    \fi%
  \else%
    \setlength{\unitlength}{\svgwidth}%
  \fi%
  \global\let\svgwidth\undefined%
  \global\let\svgscale\undefined%
  \makeatother%
  \begin{picture}(1,0.97954289)%
    \put(0,0){\includegraphics[width=\unitlength,page=1]{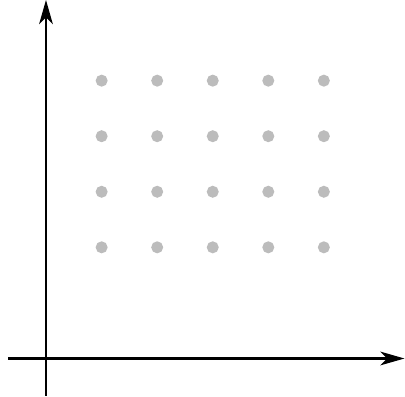}}%
    \put(-0.00356298,0.84865846){\color[rgb]{0,0,0}\makebox(0,0)[lb]{\smash{$S_2$}}}%
    \put(0.85515167,0.01422607){\color[rgb]{0,0,0}\makebox(0,0)[lb]{\smash{$S_1$}}}%
    \put(0,0){\includegraphics[width=\unitlength,page=2]{eks6grid.pdf}}%
  \end{picture}%
\endgroup%

\end{center}
\end{multicols}
\vspace{-0.2cm}
\noindent
The Lotka-Volterra system describing competitive and predator-prey interactions has been of interest for approximately a century \cite{Lotka,Volterra}. In the stochastic description of the model, we find $R=\mathbb{N}^2$ hence it follows that the state space can be divided into the endorsed and absorbing sets given by
\begin{align*}
D_E=E=\mathbb{N}^2, \qquad D_A=A=\mathbb{N}_0^2\backslash \mathbb{N}^2,
\end{align*}
respectively. Using mass-action kinetics, it follows that for $v=(v_1,v_2)\in \mathbb{N}^2$,
\begin{align*}
d_v(n)&=-\max_{x\in E, \langle v,x\rangle=n}\sum_{k=1}^r\lambda_k(x)\langle v,\xi_k\rangle\mathbbm{1}_{E}(x+\xi_k)\\
&=-\max_{x\in E, \langle v,x\rangle=n}(\alpha_1v_1 x_1+(v_2-v_1)\alpha_2 x_1x_2\mathbbm{1}_{\{2,3,\dots\}}(x_1)-v_2\alpha_3 x_2\mathbbm{1}_{\{2,3,\dots\}}(x_2)).
\end{align*}
Letting $(v_1,v_2)=(\ell,1)$ yields $d_v(n)=-\ell\alpha_1(n-1)+(\ell-1)\alpha_2(n-1)$. Thus, choosing $\ell$ sufficiently large, $d_v(n)=\mathcal{O}(n)$ and $d_v(n)>0$ for $n$ sufficiently large, provided that $\alpha_2>\alpha_1$. Note also that $\max_{k\in \mathcal{R}_A}\langle v,\xi_k \rangle<0$ hence by Proposition \ref{lem1} the process is $\mathbb{P}_x$-a.s. absorbed for all $x\in E$ if $\alpha_2>\alpha_1$. However, $d^v(n)=\mathcal{O}(n^2)$, hence Assumption \ref{ass1}(i) is not satisfied, and one can not apply Theorem \ref{res}.

\medskip
Using generalized mass-action \cite{generalized} for the same standard Lotka-Volterra network, we may obtain a different result. Suppose for example that
\begin{align*}
\lambda_1(x)=\alpha_1 x_1x_2, \quad \lambda_2(x)=\alpha_2 x_1^4 x_2^2, \quad \lambda_3(x)=\alpha_3 x_2^3.
\end{align*}
Choosing $v=(2,1)$, say, we find
\begin{align*}
d_v(n)=-\max_{x\in E, \langle v,x\rangle=n} \left(2\alpha_1 x_1x_2-\alpha_2 x_1^4 x_2^2\mathbbm{1}_{\{2,3,\dots\}}(x_2)-\alpha_3 x_2^3\mathbbm{1}_{\{2,3,\dots\}}(x_2)\right)=\mathcal{O}(n^3),
\end{align*}
and likewise
\begin{align*}
d^v(n)=\max_{x\in E, \langle v,x\rangle=n} n\left(\alpha_2 x_1^4 x_2^2 \mathbbm{1}_{\{1\}}(x_1)+\alpha_3 x_2^3 \mathbbm{1}_{\{1\}}\right)(x_2)=\mathcal{O}(n^3).
\end{align*}
Thus Assumption \ref{ass1} is satisfied. As $E$ is irreducible, Assumption \ref{class} is also satisfied and we conclude by Theorem \ref{res} that there is a unique QSD on $E$.
\begin{multicols}{2}
\begin{align}\label{local}
S_1&\underset{\alpha_2}{\overset{\alpha_1}{\rightleftharpoons}} 2S_1\nonumber\\
S_1+S_2& \underset{\alpha_4}{\overset{\alpha_3}{\rightleftharpoons}} 2S_2\\
S_2&\overset{\alpha_5}{\to} \emptyset\nonumber
\end{align}
\columnbreak
\vspace{0.4cm}
\begin{center}
 \def\svgwidth{130pt}
\begingroup%
  \makeatletter%
  \providecommand\color[2][]{%
    \errmessage{(Inkscape) Color is used for the text in Inkscape, but the package 'color.sty' is not loaded}%
    \renewcommand\color[2][]{}%
  }%
  \providecommand\transparent[1]{%
    \errmessage{(Inkscape) Transparency is used (non-zero) for the text in Inkscape, but the package 'transparent.sty' is not loaded}%
    \renewcommand\transparent[1]{}%
  }%
  \providecommand\rotatebox[2]{#2}%
  \ifx\svgwidth\undefined%
    \setlength{\unitlength}{116.48636904bp}%
    \ifx\svgscale\undefined%
      \relax%
    \else%
      \setlength{\unitlength}{\unitlength * \real{\svgscale}}%
    \fi%
  \else%
    \setlength{\unitlength}{\svgwidth}%
  \fi%
  \global\let\svgwidth\undefined%
  \global\let\svgscale\undefined%
  \makeatother%
  \begin{picture}(1,0.97954289)%
    \put(0,0){\includegraphics[width=\unitlength,page=1]{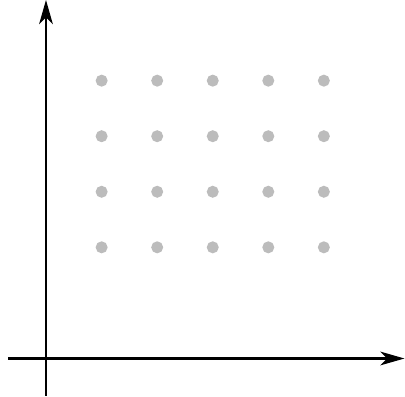}}%
    \put(-0.00356298,0.84865846){\color[rgb]{0,0,0}\makebox(0,0)[lb]{\smash{$S_2$}}}%
    \put(0.85515167,0.01422607){\color[rgb]{0,0,0}\makebox(0,0)[lb]{\smash{$S_1$}}}%
    \put(0,0){\includegraphics[width=\unitlength,page=2]{eks1grid.pdf}}%
  \end{picture}%
\endgroup%

\end{center}
\end{multicols}
\vspace{-0.3cm}
\noindent
Let us now consider the slightly altered version of the original Lotka-Volterra system using mass-action kinetics, obtained by addition of the reactions $2S_1\to S_1$ and $2S_2\to S_1+S_2$. In this case, there is still one endorsed set with corresponding absorbing set given by
\begin{align*}
E=\mathbb{N}_0\times \mathbb{N}\backslash\{(0,1)\}, \qquad A=\mathbb{N}_0^2\backslash E,
\end{align*}
respectively.

For a general $v=(v_1,v_2)\in \mathbb{N}^2$ it follows that $d^{v}(n)=\mathcal{O}(n)$ and \small
\begin{align*}
d_v(n)&=-\max_{x\in E, \langle v,x\rangle=n}\sum_{k=1}^r\lambda_k(x)\langle v,\xi_k\rangle\mathbbm{1}_{E}(x+\xi_k)\\
&=\min_{x\in E, \langle v,x\rangle=n}(-\alpha_1 x_1 v_1 +v_1\alpha_2 x_1(x_1-1)-(v_2-v_1)\alpha_3 x_1x_2\\
&\phantom{etmegetlangtmellemrum}+(v_2-v_1)\alpha_4x_2(x_2-1)+v_2\alpha_5 x_2\mathbbm{1}_{\{2,3,\dots\}}(x_2)).
\end{align*}
\normalsize
We need $v_2>v_1$ for the coefficient of the 4th reaction to be positive. Further, by the second derivative test, $d_v(n)=\mathcal{O}(n^2)$ exactly when
\begin{align*}
4v_1\alpha_2(v_2-v_1)\alpha_4>(v_2-v_1)^2\alpha_3^2.
\end{align*}
The set of possible $v$-vectors, $\mathcal{V}$, is therefore
\begin{align*}
\mathcal{V}=\left\{v\in \mathbb{N}^2\,\bigg|\, v_1<v_2<v_1\left(1+\frac{4\alpha_2\alpha_4}{\alpha_3^2}\right)\right\},
\end{align*}
which is non-empty for any positive reaction rates. A particular choice would be $v=(\ell,\ell+1)$, for $\ell$ sufficiently large. As $E$ is irreducible, Assumption \ref{class} is satisfied and we conclude by Theorem \ref{res} that the modified Lotka-Volterra system has a unique QSD on $E$ for any reaction rates.
\begin{figure}[h]
\begin{center}
\includegraphics[width=2.4in]{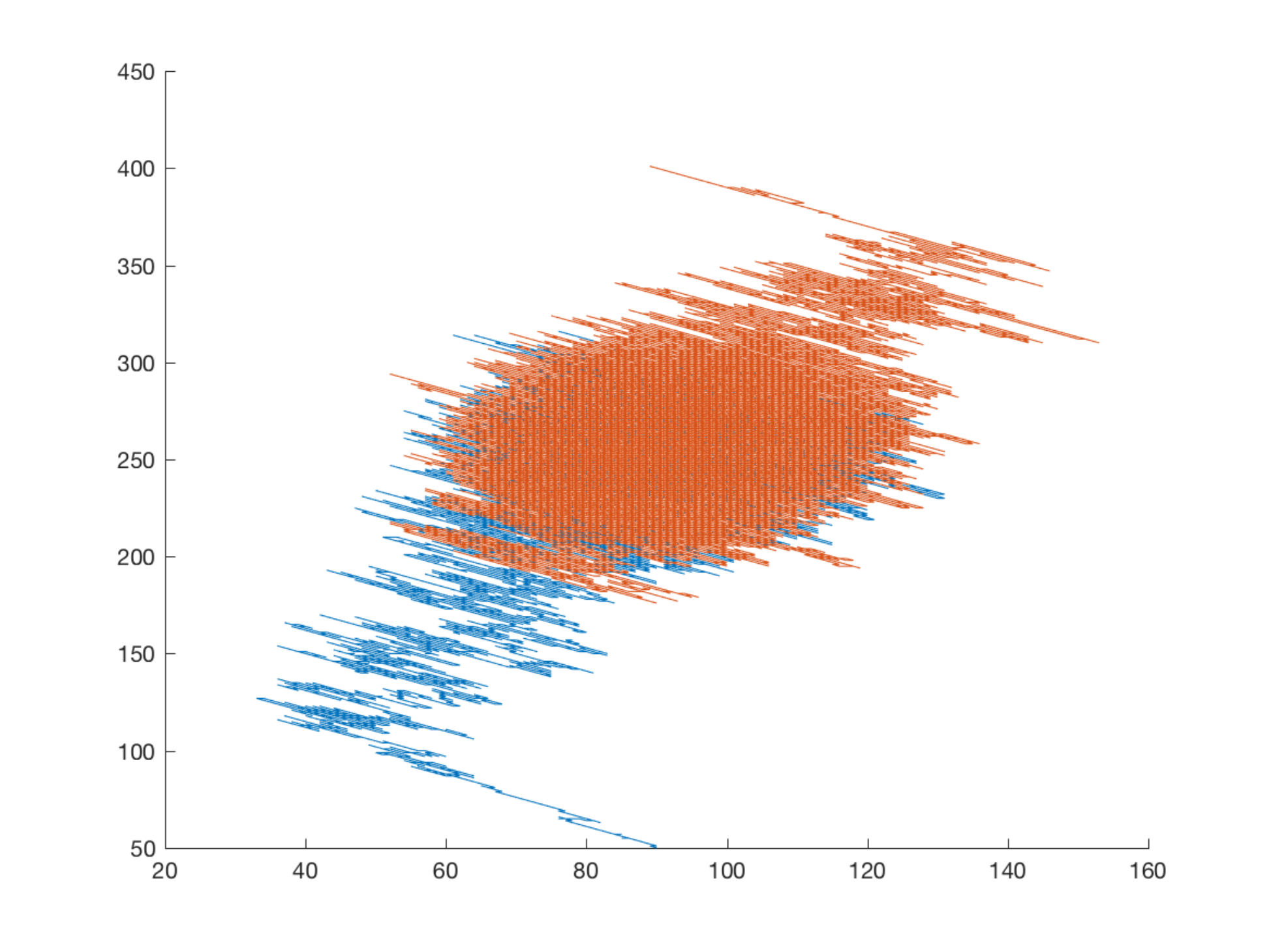}
\includegraphics[width=2.4in]{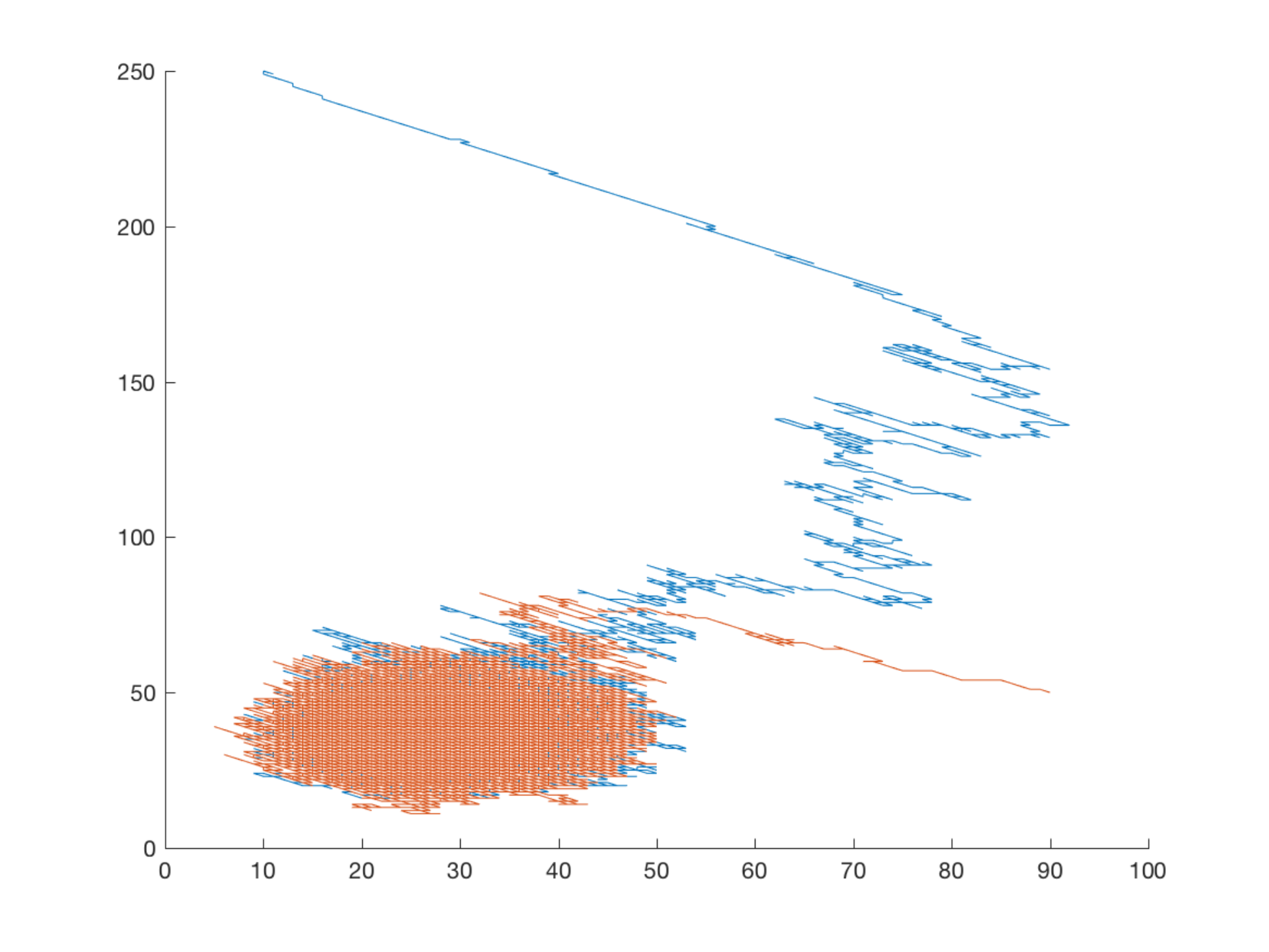}
\caption{Two realizations of the process $(X_t\colon t\geq 0)$, associated with the reaction network (\ref{local}), with differing $X_0$. Left: $\alpha_1=30,\alpha_2=0.3,\alpha_3=2,\alpha_4=0.7,\alpha_5=1$. Right: $\alpha_1=20,\alpha_2=0.7,\alpha_3=1,\alpha_4=0.7,\alpha_5=1$.}
\end{center}
\end{figure}
\end{eks}

\begin{eks}
\begin{multicols}{2}
\begin{align*}
S_1&\overset{\alpha_1}{\to} 3S_1\\
2S_2&\overset{\alpha_2}{\to} \emptyset\\
2S_1+S_2&\overset{\alpha_3}{\to} 2S_2
\end{align*}
\columnbreak
\vspace{0.4cm}
\begin{center}
 \def\svgwidth{130pt}
\begingroup%
  \makeatletter%
  \providecommand\color[2][]{%
    \errmessage{(Inkscape) Color is used for the text in Inkscape, but the package 'color.sty' is not loaded}%
    \renewcommand\color[2][]{}%
  }%
  \providecommand\transparent[1]{%
    \errmessage{(Inkscape) Transparency is used (non-zero) for the text in Inkscape, but the package 'transparent.sty' is not loaded}%
    \renewcommand\transparent[1]{}%
  }%
  \providecommand\rotatebox[2]{#2}%
  \ifx\svgwidth\undefined%
    \setlength{\unitlength}{116.48636904bp}%
    \ifx\svgscale\undefined%
      \relax%
    \else%
      \setlength{\unitlength}{\unitlength * \real{\svgscale}}%
    \fi%
  \else%
    \setlength{\unitlength}{\svgwidth}%
  \fi%
  \global\let\svgwidth\undefined%
  \global\let\svgscale\undefined%
  \makeatother%
  \begin{picture}(1,0.97954289)%
    \put(0,0){\includegraphics[width=\unitlength,page=1]{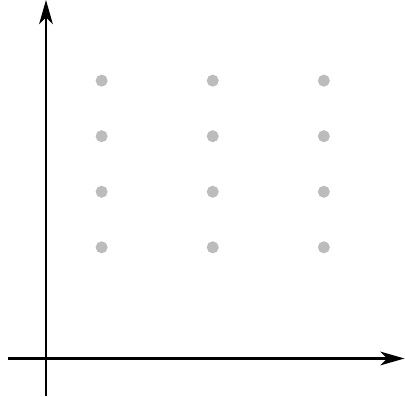}}%
    \put(-0.00356298,0.84865846){\color[rgb]{0,0,0}\makebox(0,0)[lb]{\smash{$S_2$}}}%
    \put(0.84141616,0.01422607){\color[rgb]{0,0,0}\makebox(0,0)[lb]{\smash{$S_1$}}}%
    \put(0,0){\includegraphics[width=\unitlength,page=2]{eks2grid.pdf}}%
  \end{picture}%
\endgroup%

\end{center}
\end{multicols}
\vspace{-0.5cm}
\noindent
There are two endorsed sets given by
\begin{align*}
E_1&=\{x\in \mathbb{N}^2\colon x_1=1 \mod 2\},\\
E_2&=\{x\in \mathbb{N}^2\colon x_1=0 \mod 2\}.
\end{align*}
The corresponding set of absorbing sets are given by
\begin{align*}
A_1&=\{x\in \mathbb{N}_0^d\colon x_2=0, x_1=1 \mod 2\}\\
A_2&=\{x\in \mathbb{N}_0^d\colon x_2=0, x_1=0 \mod 2\}\cup\{x\in \mathbb{N}_0^d\colon x_1=0\}.
\end{align*}
Assuming mass action kinetics, it follows that for a general $v=(v_1,v_2)\in \mathbb{N}^2$ we obtain for $E_1$ and $E_2$ respectively,
\begin{align*}
d_v(n)=-\max_{x\in E, \langle v,x\rangle =n}&\left(2v_1\alpha_1 x_1-2v_2 \alpha_2x_2(x_2-1)\mathbbm{1}_{\{3,4,\dots\}}(x_2)\right.\\
&\left.+(-2v_1+v_2)\alpha_3x_1(x_1-1)x_2\right),\\
d_v(n)=-\max_{x\in E, \langle v,x\rangle =n}&\left(2v_1\alpha_1 x_1-2v_2 \alpha_2x_2(x_2-1)\mathbbm{1}_{\{3,4,\dots\}}(x_2)\right.\\
&\left.+(-2v_1+v_2)\alpha_3x_1(x_1-1)x_2\mathbbm{1}_{\{3,4,\dots\}}(x_1)\right),
\end{align*}
which are both $\mathcal{O}(n^2)$ exactly if $v_2<2v_1$. A particular choice would be $v=(1,1)$. Further,
\begin{align*}
d^{v}(n)=\max_{x\in E, \langle v,x\rangle=n}n\big(\alpha_2 x_2(x_2-1)\mathbbm{1}_{\{2\}}(x_2)+ \alpha_3x_1(x_1-1)x_2 \mathbbm{1}_{\{2\}}(x_1)\big)=2\alpha_3n^2=\mathcal{O}(n^2).
\end{align*}
We conclude that there exists an $0<\eta<1$ such that Assumption \ref{ass1} holds. Since both $E_1$ and $E_2$ are irreducible, Assumption \ref{class} is satisfied hence Theorem \ref{samlet} applies regardless of the rate constants --  there exists a unique QSD, $\nu_n$, on each $E_n$, and given an initial distribution, $\mu$, on $D_E$, the measure $\mathbb{P}_x(X_t\in \cdot\,|\, t<\tau_A)$ approaches
\begin{align*}
\nu_\mu=\mu(E_1)\nu_1(\cdot \cap E_1)+\mu(E_2)\nu_1(\cdot \cap E_2)
\end{align*}
exponentially fast in $t$.

\begin{figure}[h]
\begin{center}
\includegraphics[width=2.9in]{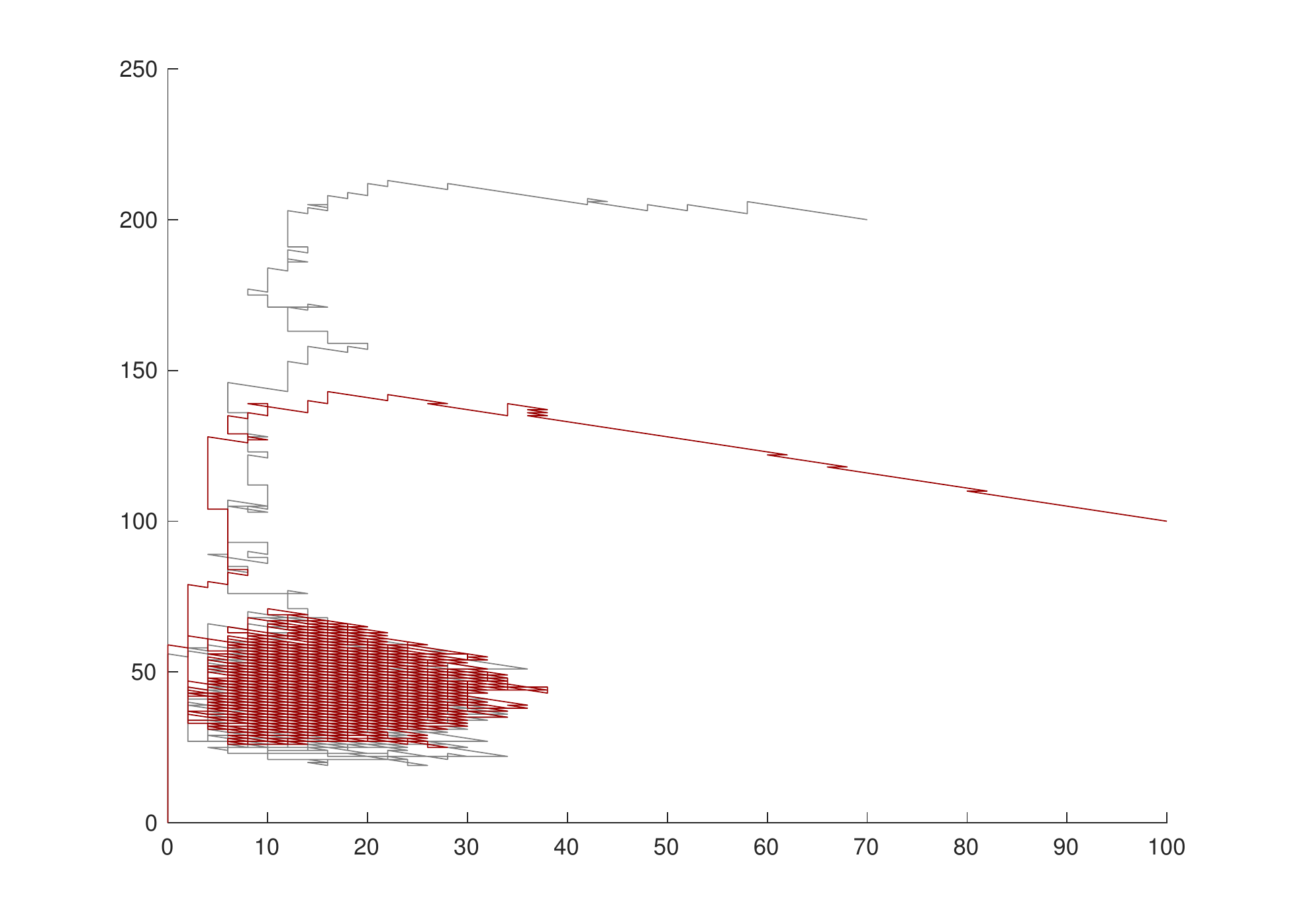}
\includegraphics[width=2.9in]{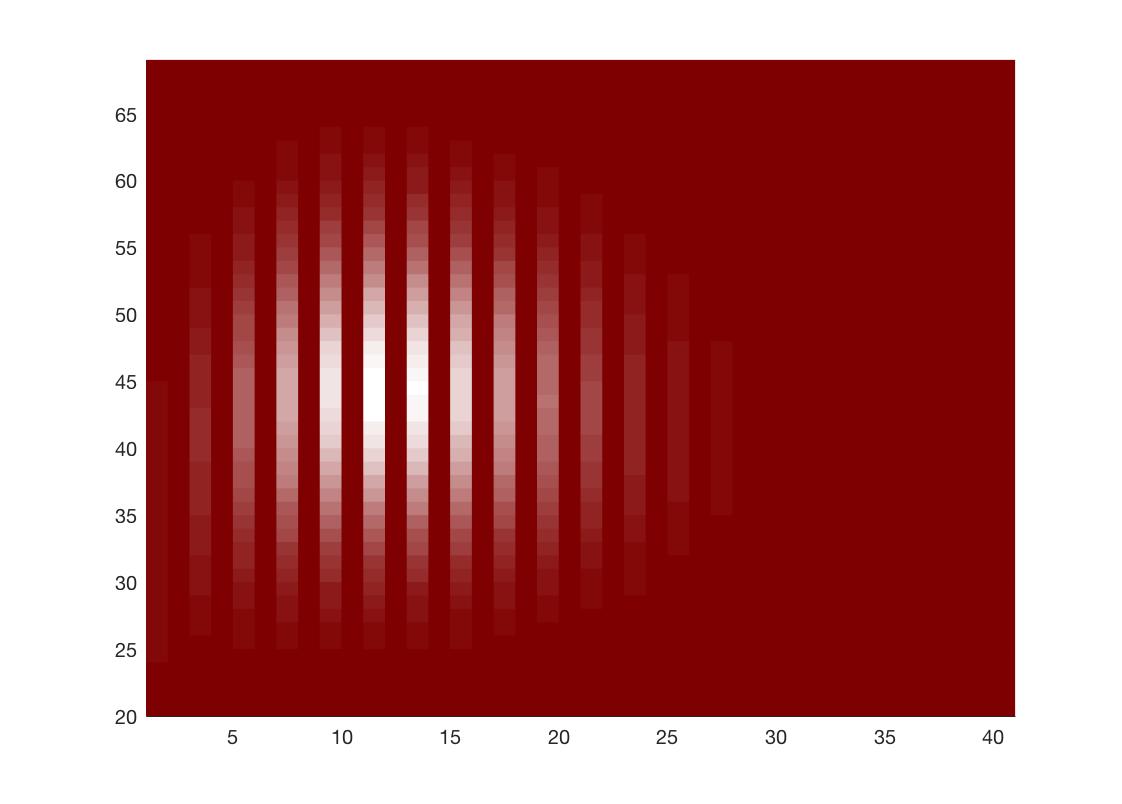}
\caption{Left: Two realizations of $(X_t\colon t\geq0)$ with $\alpha_1=300, \alpha_2=1, \alpha_3=0.5$. $X_0=(70,200)$ in grey and $X_0=(100,100)$ in red. Right: Approximate density of the QSD on one of the two endorsed sets.}
\label{trille}
\end{center}
\end{figure}

Note that if we had modeled the network using deterministic mass action \cite{crnt}, we would have found the attracting\footnote{One can calculate the trace of the Jacobin matrix to be $-2\alpha_1-(2\alpha_1\alpha_2)^{2/3}\alpha_3^{-1/3}<0$, making the fixed point attracting for all parameter values.} fixed point
\begin{align*}
(x,y)=\left(\frac{(2\alpha_1\alpha_2)^{1/3}}{\alpha_3^{2/3}},\frac{\alpha_1^{2/3}}{(2\alpha_2\alpha_3)^{1/3}}\right),
\end{align*}
which seems to lie near the peak of the quasi-stationary distribution for the parameter values used in Figure \ref{trille}. Indeed, we find the fixed point $(13.39,44.81)$ which is a stable spiral.
\end{eks}
\begin{eks}[Birth-death]
\begin{align*}
S_2 \overset{\alpha_{1}}{\leftarrow} S_1&+S_2 \overset{\alpha_{2}}{\rightarrow} 2S_1+S_2\qquad\qquad S_1 \overset{\alpha_{3}}{\leftarrow} S_1+S_2 \overset{\alpha_{4}}{\rightarrow} S_1+2S_2\\
S_1 \overset{\alpha_{5}}{\leftarrow} &2S_1 \overset{\alpha_6}{\rightarrow} 3S_1 \qquad \qquad\qquad\quad\qquad S_2 \overset{\alpha_{7}}{\leftarrow}2S_2 \overset{\alpha_{8}}{\rightarrow} 3S_2
\end{align*}
\begin{center}
 \def\svgwidth{130pt}
\begingroup%
  \makeatletter%
  \providecommand\color[2][]{%
    \errmessage{(Inkscape) Color is used for the text in Inkscape, but the package 'color.sty' is not loaded}%
    \renewcommand\color[2][]{}%
  }%
  \providecommand\transparent[1]{%
    \errmessage{(Inkscape) Transparency is used (non-zero) for the text in Inkscape, but the package 'transparent.sty' is not loaded}%
    \renewcommand\transparent[1]{}%
  }%
  \providecommand\rotatebox[2]{#2}%
  \ifx\svgwidth\undefined%
    \setlength{\unitlength}{116.48636904bp}%
    \ifx\svgscale\undefined%
      \relax%
    \else%
      \setlength{\unitlength}{\unitlength * \real{\svgscale}}%
    \fi%
  \else%
    \setlength{\unitlength}{\svgwidth}%
  \fi%
  \global\let\svgwidth\undefined%
  \global\let\svgscale\undefined%
  \makeatother%
  \begin{picture}(1,0.97954289)%
    \put(0,0){\includegraphics[width=\unitlength,page=1]{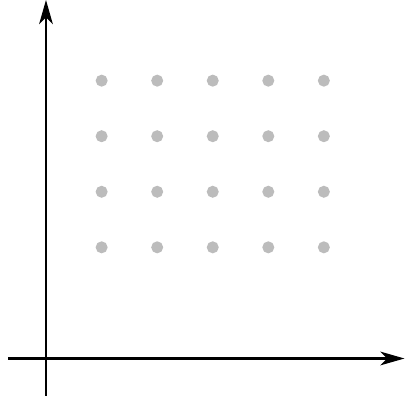}}%
    \put(-0.00356298,0.84865846){\color[rgb]{0,0,0}\makebox(0,0)[lb]{\smash{$S_2$}}}%
    \put(0.85515167,0.01422607){\color[rgb]{0,0,0}\makebox(0,0)[lb]{\smash{$S_1$}}}%
    \put(0,0){\includegraphics[width=\unitlength,page=2]{eks5grid.pdf}}%
  \end{picture}%
\endgroup%

\end{center}
\noindent
The full set of endorsed states is found to be $D_E=\mathbb{N}^2$ which is irreducible, and the full set of absorbing states is $D_A=\mathbb{N}_0^2\backslash \mathbb{N}^2$. Thus Assumption \ref{class} is satisfied. For a general $v\in \mathbb{N}^2$ we have, with $\alpha_5-\alpha_{6}=\rho_1, \alpha_7-\alpha_8=\rho_2, \alpha_{1}-\alpha_{2}=\rho_3$ and $\alpha_{3}-\alpha_{4}=\rho_4$,
\begin{align*}
d_v(n)=\min_{x\in E, \langle v,x\rangle=n} (v_1\rho_3+ v_2\rho_4)x_1x_2-v_1\rho_1x_1(x_1-1)-v_2\rho_2x_2(x_2-1).
\end{align*}
The second derivative test gives the sufficient and necessary criteria for $d_v(n)$ to be $\mathcal{O}(n^2)$,
\begin{align*}
4v_1v_2\rho_1\rho_2-(v_1\rho_3+v_2\rho_4)^2>0.
\end{align*}
Solving for $v_2$ we find the set of possible $v$-vectors,
\small
\begin{align*}
\mathcal{V}=\left\{v\in \mathbb{N}^2\,\bigg|\, \frac{2\rho_1\rho_2-\rho_3\rho_4-2\sqrt{\rho_1^2\rho_2^2-\rho_1\rho_2\rho_3\rho_4}}{\rho_4^2} v_1< v_2 <\frac{2\rho_1\rho_2-\rho_3\rho_4+2\sqrt{\rho_1^2\rho_2^2-\rho_1\rho_2\rho_3\rho_4}}{\rho_4^2} v_1\right\},
\end{align*}
\normalsize
which is non-empty exactly when $\rho_1\rho_2>\rho_3\rho_4$. Choosing parameter values as, say,
\begin{align*}
X_2 \overset{1}{\leftarrow} X_1&+X_2 \overset{3}{\rightarrow} 2X_1+X_2\qquad\qquad X_1 \overset{0.75}{\leftarrow} X_1+X_2 \overset{1}{\rightarrow} X_1+2X_2\\
X_1 \overset{2}{\leftarrow} &2X_1 \overset{1}{\rightarrow} 3X_1 \qquad \qquad\qquad\quad\qquad X_2 \overset{2}{\leftarrow}2X_2 \overset{1}{\rightarrow} 3X_2
\end{align*}
we find
\begin{align*}
\mathcal{V}\approx\{v\in \mathbb{N}^2\,|\, 1.373 v_1<v_2<46.627 v_1\}.
\end{align*}
Note that $v=(1,1)\notin \mathcal{V}$, and Theorem \ref{res} would not be applicable with this choice. We have therefore extended the result of \cite{cd}.
\end{eks}
\newpage
\begin{eks}
\begin{multicols}{2}
\begin{align*}
S_1&\overset{\alpha_1}{\to} 3S_1\\
S_1+3S_2&\overset{\alpha_2}{\to} 2S_2\\
3S_1&\overset{\alpha_3}{\to} 2S_1+3S_2
\end{align*}
\columnbreak
\vspace{0.3cm}
\begin{center}
 \def\svgwidth{125pt}
\begingroup%
  \makeatletter%
  \providecommand\color[2][]{%
    \errmessage{(Inkscape) Color is used for the text in Inkscape, but the package 'color.sty' is not loaded}%
    \renewcommand\color[2][]{}%
  }%
  \providecommand\transparent[1]{%
    \errmessage{(Inkscape) Transparency is used (non-zero) for the text in Inkscape, but the package 'transparent.sty' is not loaded}%
    \renewcommand\transparent[1]{}%
  }%
  \providecommand\rotatebox[2]{#2}%
  \ifx\svgwidth\undefined%
    \setlength{\unitlength}{116.48636904bp}%
    \ifx\svgscale\undefined%
      \relax%
    \else%
      \setlength{\unitlength}{\unitlength * \real{\svgscale}}%
    \fi%
  \else%
    \setlength{\unitlength}{\svgwidth}%
  \fi%
  \global\let\svgwidth\undefined%
  \global\let\svgscale\undefined%
  \makeatother%
  \begin{picture}(1,0.97954289)%
    \put(0,0){\includegraphics[width=\unitlength,page=1]{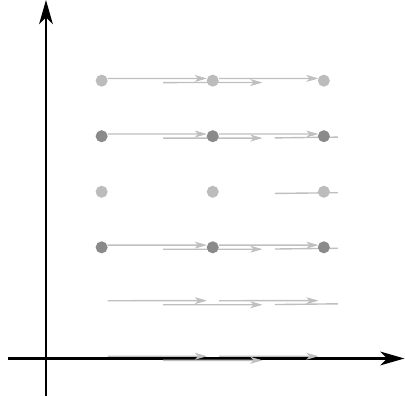}}%
    \put(-0.00356298,0.84865846){\color[rgb]{0,0,0}\makebox(0,0)[lb]{\smash{$S_2$}}}%
    \put(0.84141616,0.01422607){\color[rgb]{0,0,0}\makebox(0,0)[lb]{\smash{$S_1$}}}%
    \put(0,0){\includegraphics[width=\unitlength,page=2]{eks3grid.pdf}}%
  \end{picture}%
\endgroup%

\end{center}
\end{multicols}
\vspace{-0.5cm}
\noindent
The stoichiometric matrix is in this example given by
\begin{align*}
\Xi=\begin{pmatrix}
2 & -1 & -1\\
0 & -1 & 3
\end{pmatrix}.
\end{align*}
It follows that $\rank \Xi=2=d$ hence by Proposition \ref{state}, there is a finite number endorsed sets for $x$ sufficiently large. Indeed, upon inspection, we find two endorsed sets given by
\begin{align*}
E_1&=\{x\in \mathbb{N}^2: x_2\geq 0, x_1+x_2=0 \mod 2\},\\
E_2&=\{x\in \mathbb{N}^2: x_2\geq 0, x_1+x_2=1 \mod 2\},
\end{align*}
and these each have a unique minimal irreducible class hence Assumption \ref{class} is satisfied. Note, however, that the sets $E_1$ and $E_2$ are not irreducible. Here $D_E=\{x\in \mathbb{N}_0^2: x_1\geq 1\}$ and $D_A=\{x\in\mathbb{N}_0^2: x_1=0\}$. Taking $v=(4,1)$, it follows that for both classes
\begin{align*}
d_v(n)&=-\max_{x\in E, \langle v,x\rangle =n}(8\alpha_1 x_1-5 \alpha_2x_1x_2(x_2-1)(x_2-2)\mathbbm{1}_{\{2,3,\dots\}}(x_1)-\alpha_3x_1(x_1-1)(x_1-2)),
\end{align*}
which is $\mathcal{O}(1)$. Furthermore,
\begin{align*}
d^{v}(n)=\max_{x\in E, \langle v,x\rangle=n}n x_1x_2(x_2-1)(x_2-2)\mathbbm{1}_{\{1\}}(x_1)=\mathcal{O}(n^4).
\end{align*}
Thus, Assumption \ref{ass1} is not satisfied and we can not apply Theorem \ref{samlet}.
\end{eks}

\begin{eks}
\begin{multicols}{2}
\begin{align*}
S_1+S_2&\overset{\alpha_1}{\to} \emptyset\\
S_1&\overset{\alpha_2}{\to} 2S_1+S_2\\
2S_1+2S_2&\overset{\alpha_3}{\to} S_1+S_2
\end{align*}
\columnbreak
\vspace{0.4cm}
\begin{center}
 \def\svgwidth{129pt}
\begingroup%
  \makeatletter%
  \providecommand\color[2][]{%
    \errmessage{(Inkscape) Color is used for the text in Inkscape, but the package 'color.sty' is not loaded}%
    \renewcommand\color[2][]{}%
  }%
  \providecommand\transparent[1]{%
    \errmessage{(Inkscape) Transparency is used (non-zero) for the text in Inkscape, but the package 'transparent.sty' is not loaded}%
    \renewcommand\transparent[1]{}%
  }%
  \providecommand\rotatebox[2]{#2}%
  \ifx\svgwidth\undefined%
    \setlength{\unitlength}{116.48636904bp}%
    \ifx\svgscale\undefined%
      \relax%
    \else%
      \setlength{\unitlength}{\unitlength * \real{\svgscale}}%
    \fi%
  \else%
    \setlength{\unitlength}{\svgwidth}%
  \fi%
  \global\let\svgwidth\undefined%
  \global\let\svgscale\undefined%
  \makeatother%
  \begin{picture}(1,0.97954289)%
    \put(0,0){\includegraphics[width=\unitlength,page=1]{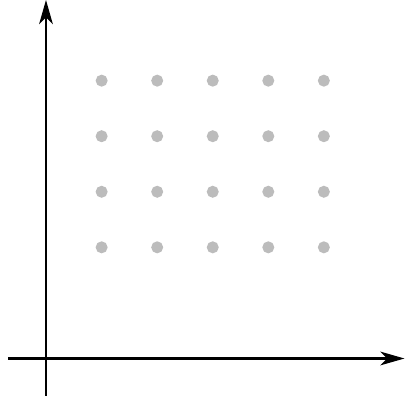}}%
    \put(-0.00356298,0.84865846){\color[rgb]{0,0,0}\makebox(0,0)[lb]{\smash{$S_2$}}}%
    \put(0.85515167,0.01422607){\color[rgb]{0,0,0}\makebox(0,0)[lb]{\smash{$S_1$}}}%
    \put(0,0){\includegraphics[width=\unitlength,page=2]{eks88grid.pdf}}%
  \end{picture}%
\endgroup%

\end{center}
\end{multicols}
\vspace{-0.3cm}
\noindent
As $\rank \Xi=1<d$, there are infinitely many endorsed sets. These are each irreducible. Indeed, upon inspection, we find the endorsed sets
\begin{align*}
E_i=\{x\in \mathbb{N}_0^2: x_1\geq 1, x_2=x_1+(-1)^i\left \lfloor{i/2}\right \rfloor \}, \qquad i\geq 1,
\end{align*}
and the full set of absorbing states $D_A=\{x\in\mathbb{N}_0^2: x_1=0\}$. Furthermore, some sets ($E_i$ with $i\equiv1\mod 2$) have no corresponding absorbing set while all others do. For all classes however, we find
\begin{align*}
d_v(n)=\mathcal{O}(n^2), \qquad d^v(n)=\mathcal{O}(n),
\end{align*}
hence, by Theorem \ref{samlet}, we conclude that for any initial distribution $\mu$ with support contained in a finite subset of the endorsed sets, the measure $\mathbb{P}_\mu(X_t\in \cdot \,|\, t<\tau_A)$ converges exponentially fast in $t$ to $\nu_\mu$. Had we instead looked at the slightly altered network
\vspace{-0.2cm}
\begin{multicols}{2}
\begin{align*}
S_1+S_2&\overset{\alpha_1}{\to} \emptyset\\
S_1&\overset{\alpha_2}{\to} 2S_1+S_2\\
2S_1+2S_2&\overset{\alpha_3}{\to} S_1+S_2\\
3S_1+S_2&\overset{\alpha_4}{\to}  3S1
\end{align*}
\columnbreak
\vspace{0.4cm}
\begin{center}
 \def\svgwidth{129pt}
\begingroup%
  \makeatletter%
  \providecommand\color[2][]{%
    \errmessage{(Inkscape) Color is used for the text in Inkscape, but the package 'color.sty' is not loaded}%
    \renewcommand\color[2][]{}%
  }%
  \providecommand\transparent[1]{%
    \errmessage{(Inkscape) Transparency is used (non-zero) for the text in Inkscape, but the package 'transparent.sty' is not loaded}%
    \renewcommand\transparent[1]{}%
  }%
  \providecommand\rotatebox[2]{#2}%
  \ifx\svgwidth\undefined%
    \setlength{\unitlength}{116.48636904bp}%
    \ifx\svgscale\undefined%
      \relax%
    \else%
      \setlength{\unitlength}{\unitlength * \real{\svgscale}}%
    \fi%
  \else%
    \setlength{\unitlength}{\svgwidth}%
  \fi%
  \global\let\svgwidth\undefined%
  \global\let\svgscale\undefined%
  \makeatother%
  \begin{picture}(1,0.97954289)%
    \put(0,0){\includegraphics[width=\unitlength,page=1]{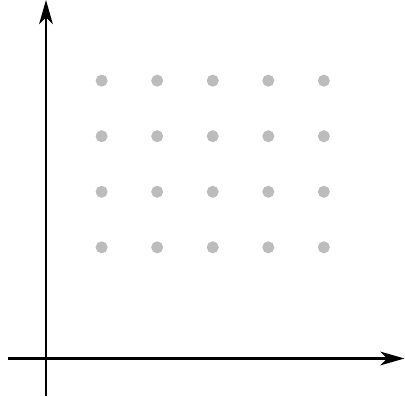}}%
    \put(-0.00356298,0.84865846){\color[rgb]{0,0,0}\makebox(0,0)[lb]{\smash{$S_2$}}}%
    \put(0.85515167,0.01422607){\color[rgb]{0,0,0}\makebox(0,0)[lb]{\smash{$S_1$}}}%
    \put(0,0){\includegraphics[width=\unitlength,page=2]{eks89grid.pdf}}%
  \end{picture}%
\endgroup%

\end{center}
\end{multicols}
\vspace{-0.3cm}
\noindent
then Assumption \ref{class} would no longer be satisfied as no minimal irreducible class exists.
\end{eks}

\section{Conclusions}
In this paper, we have provided sufficient conditions for the existence and uniqueness of a quasi-stationary distribution, within each endorsed set, for general stochastic reaction networks. In particular, we have provided sufficient conditions for the existence and uniqueness of a globally attracting quasi-stationary distribution in the space of probability distributions.

The requirement that for mass-action reaction systems there exists a stoichiometric coefficient strictly greater than one for each species is strong, however, it seems to be intrinsic to the problem of guaranteeing uniqueness of the QSD. Indeed, this can be seen already in 1-dimensional systems, and one can imagine that getting stuck near one of the axes would approximately reduce the multi-dimensional system to such a 1-dimensional case. It is an important question to determine sufficient conditions for just the existence of a QSD, in which case one would expect much weaker conditions to be satisfied. This is still an open problem.

The application of our results depends on the existence of a vector $v\in \mathbb{N}^d$, such that Assumption \ref{ass1} holds. One would like to have easy graphical ways of guaranteeing this existence or an explicit algorithmic way of constructing such a vector. This may not be possible in general, since the problem is equivalent to determining the sign of a multivariate polynomial in a certain region of the positive orthant. However, the specific network at hand is often prone to analytical ad hoc methods, and even in lack of this, one may rely on numerical methods to find a suitable candidate for $v$. Another caveat is the exact calculation of the endorsed sets. Methods for finding these numerically, in the case where the intensity functions are positive in the positive orthant, exist \cite{gham,craciun}. However, the problem becomes more complicated in the general case \cite{scale}.

Knowing the existence of a unique QSD, one would of course like to know the explicit analytical expression for this distribution on each endorsed set. However, this seems to be a very hard problem to solve in general, and even for 1-dimensional systems it has not yet been fully resolved. Thus, so far, one is forced to apply numerical methods or rely on analytical approximations, see \cite{simu,mini}.

Another problem which is important for applications, is the question of observability. Indeed, the relative sizes of the time to extinction, $\tau_{A}$, versus the time to reach the mean of the QSD, $\tau_\nu$, determines whether we are likely to observe the QSD or not. Only if $\tau_{A}\gg \tau_{\nu}$ would we expect the process to behave according to the QSD \cite{vellela}. Few recent results on this matter exist for limited scenarios \cite{scale}, while methods exploiting the WKB approximation \cite{WKB} are more general although not yet fully rigorous \cite{Chazottes}.

Finally, numerical evidence seems to suggest a strong connection between the deterministic and stochastic models of the same underlying reaction network. Indeed, for systems close to thermodynamic equilibrium, also referred to as the fluid limit \cite{kurtz}, the modes of the QSD appear to be located near the deterministic steady states. However, far from equilibrium, the picture may be radically different. Our result can be seen as a stochastic analogue to the deterministic case of having an equilibrium point within each stoichiometric compatibility class. In this light, the QSD bridges the gap between the knowledge of extinction in the stochastic description and the existence of a stationary steady state in the deterministic setting. Future work lies in analyzing what can be inferred about the QSD from the corresponding deterministic dynamical system.
\vspace{-0.15cm}
\section*{Acknowledgement}
The authors would like to thank Abhishek Pal Majumder for many valuable and interesting discussions on this topic.
\newpage
\bibliographystyle{siam}

\end{document}